\documentclass[12pt]{amsart}

\usepackage{amssymb,latexsym,amsthm, amsmath, comment, fullpage, wrapfig}
\usepackage[all]{xy}
\usepackage{graphicx, comment,  tikz-cd, libertine}

\ifx\pdfoutput\undefined
\usepackage{hyperref}
\else
\usepackage[pdftex,colorlinks=true,linkcolor=blue,urlcolor=blue]{hyperref}
\fi

\usepackage{mathtools,bm,booktabs,enumitem,microtype,longtable,array}
\usepackage[nameinlink,capitalise,noabbrev]{cleveref}

\hypersetup{
  citecolor=blue,
  pdftitle={Failure of entropic selection for diagonal matrix completion in deep linear networks},
  pdfauthor={Rodrigo Trevino},
  pdfsubject={Deep linear networks, matrix completion, entropy, free-energy gradient flow}
}

\setlist[itemize]{topsep=4pt,itemsep=2pt,parsep=1pt}
\setlist[enumerate]{topsep=4pt,itemsep=3pt,parsep=1pt}

\newtheorem{theorem}{Theorem}[section]
\newtheorem{proposition}[theorem]{Proposition}
\newtheorem{lemma}[theorem]{Lemma}
\newtheorem{corollary}[theorem]{Corollary}

\theoremstyle{definition}
\newtheorem{definition}[theorem]{Definition}

\theoremstyle{remark}
\newtheorem{remark}[theorem]{Remark}

\newcommand{\Md}{\mathbb{M}_d}
\newcommand{\GLd}{\mathrm{GL}_d(\mathbb{R})}

\newcommand{\diag}{\operatorname{diag}}

\newcommand{\grad}{\operatorname{grad}}
\newcommand{\Tr}{\operatorname{Tr}}
\newcommand{\rank}{\operatorname{rank}}
\newcommand{\Crit}{\operatorname{Crit}}
\newcommand{\sgn}{\operatorname{sgn}}
\newcommand{\Aop}{\mathcal{A}}
\newcommand{\Sn}{\mathcal{S}_N}
\newcommand{\cN}{c_N}
\newcommand{\eps}{\varepsilon}
\newcommand{\ip}[2]{\left\langle #1,#2\right\rangle_F}
\newcommand{\normF}[1]{\left\lVert #1\right\rVert_F}
\newcommand{\normg}[1]{\left\lVert #1\right\rVert_{g^N}}

\title{Failure of entropic selection for diagonal matrix completion in deep linear networks}
\author{Rodrigo Trevi\~no}
\date{\today}

\begin{document}

\begin{abstract}
We study the free-energy gradient flow associated with the depth-$N$ deep linear network on the space of invertible real $d\times d$ matrices.  The loss observes only the diagonal,
\[
E_d(W)=\frac12\sum_{i=1}^d(W_{ii}-1)^2,
\]
and the regularizer is the Boltzmann entropy of the balanced factorization fiber computed by Menon and Yu.  This is a natural higher-dimensional version of a matrix-completion problem posed by Menon as a test of whether entropy selects among a noncompact family of minimizers.

We first isolate the general obstruction to such a selection principle.  Up to a uniformly bounded term, the entropy is a logarithmic sum of the maximal exterior-volume expansions of $W$.  This yields an exact criterion for boundedness of $E-\beta^{-1}S_N$ in terms of the maximum entropy on energy sublevel sets, as well as practical sufficient conditions based on logarithmic, superlogarithmic, and polynomial growth of the energy.  We prove that the entropy tends to $+\infty$ along every nontrivial affine ray through a full-rank matrix.  Consequently, for a squared linear matrix-sensing loss, the free energy is bounded below if and only if the observation operator is injective; in the injective case it is coercive in norm.  We also distinguish norm coercivity from attainment on the open full-rank manifold and give a compactness criterion under which low-temperature minimizers select entropy maximizers on the minimizing set of the energy.

For diagonal completion, every width $d\geq2$, depth $N>2$, and inverse temperature $\beta>0$, we prove that the free energy is unbounded below and has exactly $2^d$ full-rank critical points, one in each diagonal sign chamber.  Every critical point is diagonal and hyperbolic.  Its unstable dimension is $\binom d2$ and its stable dimension is $d(d+1)/2$.  We also prove that a full-rank trajectory cannot converge to a finite rank-deficient matrix.  Corank at least two is excluded by an infinite free-energy barrier, while the corank-one boundary is repelling because the entropy produces a positive linear drift in the smallest singular value.  The flow is forward complete.  Every bounded orbit converges to one of the $2^d$ saddles, and the union of their stable sets has Lebesgue measure zero.  Consequently, almost every full-rank initial condition has an unbounded forward orbit.

To separate this failure from the low-rank selection displayed by ordinary training, we also analyze the unregularized flow at width two.  No fixed full-rank initial condition can have a nonzero rank-one completion as its terminal state.  Nevertheless, for explicit scaled families $W_0^{\eps}=\eps\widehat W_0$ with fixed full-rank $\widehat W_0$, the terminal matrix remains full rank for every sufficiently small $\eps>0$ but converges to a rank-one completion as $\eps\to0$.

Thus finite-temperature fiber entropy does not provide an equilibrium selection principle for diagonal completion, or more generally for any genuinely underdetermined linear matrix-sensing problem on the unrestricted matrix space.
\end{abstract}


\maketitle
\newpage
\setcounter{tocdepth}{1}
\tableofcontents

\section{Introduction}

A deep linear network represents a matrix $W$ as a product of $N$ trainable factors.  On balanced factorizations, Euclidean gradient flow in factor space projects to a Riemannian gradient flow for the end-to-end matrix.  The metric depends on the depth and can therefore alter which minimizer is reached when a loss has a degenerate minimizing set.  This provides a finite-dimensional model for implicit regularization in overparametrized learning; see \cite{AroraCohenHazan2018,BahEtAl2022,CohenMenonVeraszto2023,Menon2025}.

Menon and Yu introduced a Boltzmann entropy $S_N$ for the deep linear network by taking the logarithm of the volume of the balanced parameter-space fiber over an end-to-end matrix \cite{MenonYu2026}.  Given a loss (or energy) $E$ and inverse temperature $\beta>0$, this leads to the free energy
\begin{equation*}
F_\beta(W)=E(W)-\frac1\beta S_N(W),
\end{equation*}
where $S_N(W)$ is the entropy, and to the Riemannian gradient flow
\begin{equation}
  \label{eq:gradFlow}
\dot W=-\grad_{g^N}F_\beta(W).
\end{equation}
For spectral confining losses, Chen, Kotwal, and Menon characterize the equilibria and dynamics of this flow and prove strict concavity of the entropy in singular-value coordinates \cite{ChenKotwalMenon2025}.  The nonspectral case, including matrix completion, is explicitly left open there.

The motivating completion loss is
\begin{equation}\label{eq:completion-loss-intro}
E_d(W)=\frac12\sum_{i=1}^d(W_{ii}-1)^2.
\end{equation}
Its minimizers form the noncompact affine space of matrices with diagonal equal to $1$.  Menon posed the $2\times2$ version both as a low-rank selection problem for the ordinary deep-linear flow and as a test of whether the free energy selects among degenerate minimizers \cite[Sections 14.1--14.2]{Menon2025}.  Numerical work on the unregularized flow shows attraction toward low-rank completions in small examples \cite{CohenMenonVeraszto2023}.

The main conclusion of this paper is that the free-energy minimization proposal fails structurally for \eqref{eq:completion-loss-intro}.  The free energy has no minimizer, all of its finite full-rank equilibria are saddles, finite rank-deficient limits are impossible, and generic trajectories are unbounded.

To separate this finite-temperature failure from the low-rank behavior of ordinary training, \cref{sec:zero-temperature} analyzes the unregularized flow at width two.  We prove that a full-rank trajectory cannot converge exactly to a nonzero rank-one completion.  We then exhibit scaled full-rank initializations $W_0^{\eps}=\eps\widehat W_0$, with fixed full-rank $\widehat W_0$, whose terminal matrices are full rank for every sufficiently small $\eps>0$ but converge to a rank-one completion as $\eps\to0$.  On a symmetric invariant plane the endpoint can be computed exactly: at every depth $N>2$ the terminal smaller singular value is asymptotic to an explicit positive multiple of $\eps$, while at depth two the mixed-sign family has an explicit positive multiple of $\eps^2$ and the same-sign family retains a nondegenerate rank-two limit.  We also prove persistence under a nonlinear nonsymmetric equal-diagonal perturbation; when the initial singular-frame mismatch is nonzero, the difference of the terminal off-diagonal entries is asymptotic to an explicit nonzero multiple of $\eps^{2/N}$.  The same-sign symmetric-plane endpoints recover, in width two, a recent result of Shin and Yun \cite{ShinYun2026}; see \cref{rem:SY-correspondence}.

The failure is not caused merely by degeneracy of the minimizing set.  A compact family of minimizers can support a genuine low-temperature selection principle: provided minimizers of $F_\beta$ exist and remain precompact, their limit points maximize the entropy on the set of energy minimizers.  What is fatal here is the combination of noncompact low-energy directions and unbounded entropy along those directions.

The entropy growth is weak in one sense and strong in another.  It is only logarithmic in the size of the matrix, so any superlogarithmically confining energy dominates it.  But it diverges along every nontrivial affine ray through a full-rank matrix.  Therefore any loss that is bounded on such a ray has free energy unbounded below.  This applies in particular to every noninjective squared linear matrix-sensing problem on the unrestricted matrix space.  In \cref{sec:general-energy} we prove the sharp dichotomy
\[
\frac12\|\mathcal L(W)-y\|_2^2-\frac1\beta S_N(W)
\text{ is bounded below}
\mbox{ if and only if }
\mathcal L\text{ is injective}.
\]
When $\mathcal L$ is injective, the quadratic sensing loss dominates the logarithmic entropy and the free energy is coercive in norm.  We also formulate an exact criterion in terms of the entropy available on energy sublevel sets and explain what additional control near the rank boundary is needed for a minimizer to be attained inside $\GLd$.

\subsection{Main results}

We reserve $d$ for the matrix width and $N$ for the network depth.  The entropy $S_N$ is defined precisely in \cref{sec:setup}. The following are the main results of the paper.
\begin{theorem}
  \label{thm:main}
Fix $d\geq2$, an integer depth $N>2$, and $\beta>0$.  Let
\[
F_\beta(W)=\frac12\sum_{i=1}^d(W_{ii}-1)^2-\frac1\beta S_N(W),
\qquad W\in\GLd,
\]
and let $W(t)$ solve the depth-$N$ free-energy flow \eqref{eq:gradFlow}.  Then the following statements hold.

\begin{enumerate}[label=\textup{(\roman*)}]
\item The free energy is unbounded below:
\[
\inf_{W\in\GLd}F_\beta(W)=-\infty.
\]

\item There are exactly $2^d$ critical points in $\GLd$.  Every critical point is diagonal.  More precisely, for each sign vector
\[
\epsilon=(\epsilon_1,\ldots,\epsilon_d)\in\{-1,1\}^d
\]
there is a unique $x^\epsilon\in(0,\infty)^d$ such that
\[
W_\epsilon=\diag(\epsilon_1x_1^\epsilon,\ldots,\epsilon_dx_d^\epsilon)
\]
is critical, and these are all the critical points.

\item Every $W_\epsilon$ is a nondegenerate saddle. In addition, the stable and unstable spaces of the linearized free-energy flow satisfy
\[
\dim E^s=\frac{d(d+1)}2,
\qquad
\dim E^u=\binom d2.
\]

\item No solution starting in $\GLd$ can converge to a finite rank-deficient matrix.

\item Every full-rank solution is defined for all $t\geq0$ and remains in its initial connected component of $\GLd$.

\item Every bounded forward orbit converges to one of the $2^d$ critical points.  The set of initial conditions with bounded forward orbit has Lebesgue measure zero in $\Md$.  Consequently, for Lebesgue-almost every $W_0\in\GLd$,
\[
\sup_{t\geq0}\normF{W(t)}=\infty.
\]
\end{enumerate}
\end{theorem}

The theorem does not assert that $\normF{W(t)}$ has a limit or tends monotonically to infinity.  It also does not exclude an unbounded trajectory for which some singular values tend to zero while others diverge.  We call this possible behavior \emph{singular escape at infinity}.  The theorem excludes only convergence to a finite singular matrix.

\subsection{Mechanism and proof strategy}

Four elementary mechanisms drive the proof.

First, a zero-loss shear has one singular value tending to infinity.  The entropy grows logarithmically along this family, so the free energy tends to $-\infty$.

Second, strict concavity of the entropy in singular-value coordinates forces the polar factor of a critical matrix to align with the coordinate axes.  This makes every critical point diagonal.  On each diagonal sign chamber the free energy becomes a strongly convex function of the positive diagonal magnitudes, giving exactly one point per sign vector.

Third, the full matrix Hessian is indefinite at each of these diagonal points.  Every unordered coordinate pair contributes one stable and one unstable off-diagonal mode.  The $d$ diagonal modes are stable.  This produces a Morse index $\binom d2$.

Fourth, the entropy has two different effects near the rank boundary.  If at least two singular values vanish, the entropy tends to $-\infty$, so the free energy tends to $+\infty$ and cannot be approached by a descending trajectory.  If exactly one singular value vanishes, the entropy instead creates a positive drift of order $\sigma_d$, while the energy contribution is of order $\sigma_d^{2-2/N}=o(\sigma_d)$ when $N>2$.  The corank-one boundary is therefore repelling.

The remainder of the paper makes these points precise and then uses the Lyapunov identity, compactness, LaSalle's principle, and the stable-manifold theorem to obtain generic unboundedness.

\subsection{Relation to prior work and scope}

Implicit regularization in matrix factorization has been studied from several complementary viewpoints.  At depth two, nuclear-norm behavior was identified under important symmetry and commuting-measurement assumptions by Gunasekar et al.\ \cite{GunasekarEtAl2017}.  Arora et al.\ studied the additional bias created by deeper factorizations and emphasized that it is not generally captured by a fixed norm \cite{AroraCohenHuLuo2019}.  Cohen, Menon, and Veraszto developed the Riemannian description of deep-linear matrix completion and supplied rigorous and numerical evidence relating its geometry to low-rank outcomes \cite{CohenMenonVeraszto2023}.  More recently, Shin and Yun studied the factor gradient flow started from a balanced deterministic family of initializations under block-diagonal observations, derived exact equations for the limiting singular values, and showed that at depth at least three the stable rank of the limit tends to one as the initialization scale tends to zero, in contrast with depth two \cite{ShinYun2026}.  In width two their diagonal case coincides with the symmetric-plane analysis of \cref{sec:zero-temperature}; see \cref{rem:SY-correspondence}.

All of those low-rank statements concern the ordinary, unregularized training flow, which is the formal $\beta=\infty$ limit of the system studied here.  Menon and Yu subsequently computed the balanced-fiber entropy \cite{MenonYu2026}, and Chen, Kotwal, and Menon analyzed the resulting free energy for spectral losses, proving in particular the strict concavity theorem used below \cite{ChenKotwalMenon2025}.  They explicitly left nonspectral losses such as matrix completion open.

Our principal contribution concerns the finite-$\beta$ free-energy flow for the simplest diagonal observation pattern.  Its conclusions do not contradict low-rank selection by the ordinary flow.  Rather, they show that an arbitrarily small but fixed entropic drift changes the global long-time problem.  The width-two analysis in \cref{sec:zero-temperature} makes the contrast explicit: at $\beta=\infty$, there are scaled full-rank families whose terminal smaller singular value is asymptotic to an explicit positive multiple of $\eps$ for $N>2$ and whose terminal matrices converge to rank-one completions as $\eps\to0$, while at every fixed finite $\beta$ the generic orbit is unbounded.  We work throughout with the reduced balanced dynamics and use the full-rank fiber entropy only in the finite-temperature problem.

\vspace{.1in}

\noindent \textbf{Tool and computational resource disclosure:} Following the Leiden Declaration's recommendations, I am disclosing that ChatGPT Sol and Claude Fable were used in the preparation of this paper.

\vspace{.1in}

\noindent \textbf{Acknowledgements:} This work was supported by a Career Transition Grant from Coefficient Giving.

\section{Deep-linear geometry and the entropy}\label{sec:setup}
Most of the material in this section can be gathered from the great survey \cite{Menon2025}.
\subsection{The balanced end-to-end flow}
Let $\Md$ denote the real $d\times d$ matrices, equipped with the Frobenius inner product
\[
\ip{P}{Q}=\Tr(P^TQ).
\]
A depth-$N$ deep linear network is an $N$-tuple $(W_N,\ldots,W_1)$ representing the end-to-end matrix
\[
W=W_NW_{N-1}\cdots W_1.
\]
Balanced factorizations satisfy
\[
W_{p+1}^TW_{p+1}=W_pW_p^T,
\qquad p=1,\ldots,N-1.
\]
They form invariant manifolds for factor-space gradient flow.  On the full-rank balanced manifold, the end-to-end dynamics close on $\GLd$ and are determined by the positive self-adjoint operator
\begin{equation*}
\Aop_{N,W}(P)
=
\sum_{k=1}^N
(WW^T)^{\frac{N-k}{N}}
P
(W^TW)^{\frac{k-1}{N}}.
\end{equation*}
For a smooth function $H:\GLd\to\mathbb R$, the corresponding Riemannian gradient is
\begin{equation*}
\grad_{g^N}H(W)=\Aop_{N,W}(H'(W)),
\end{equation*}
where $H'(W)$ is the Euclidean gradient.  These facts follow from the balanced reduction and Riemannian-submersion results in \cite{BahEtAl2022,Menon2025,MenonYu2026}.

If
\[
W=U\Sigma V^T,
\qquad
\Sigma=\diag(\sigma_1,\ldots,\sigma_d),
\qquad
\sigma_i>0,
\]
then the singular frame $u_iv_j^T$ diagonalizes $\Aop_{N,W}$:
\begin{equation}\label{eq:A-spectrum}
\Aop_{N,W}(u_iv_j^T)=a_{ij}(\sigma)u_iv_j^T,
\end{equation}
where
\begin{equation*}
a_{ii}(\sigma)=N\sigma_i^{2-2/N}
\end{equation*}
and, for $i\neq j$,
\begin{equation}\label{eq:aij-off}
a_{ij}(\sigma)
=
\frac{\sigma_i^2-\sigma_j^2}
{\sigma_i^{2/N}-\sigma_j^{2/N}}.
\end{equation}
The quotient in \eqref{eq:aij-off} is interpreted by continuity when $\sigma_i=\sigma_j$.  It is positive.  Thus $\Aop_{N,W}$ is positive definite on $\GLd$. The Riemannian metric is
\begin{equation*}
g_W^N(P,Q)=\ip{P}{\Aop_{N,W}^{-1}(Q)}.
\end{equation*}
Given the free energy $F_\beta$, the flow studied in this paper is therefore
\begin{equation}\label{eq:free-flow}
\dot W
=-\Aop_{N,W}\bigl(F_\beta'(W)\bigr)
=-\Aop_{N,W}\bigl(E_d'(W)\bigr)
+\frac1\beta\Aop_{N,W}\bigl(S_N'(W)\bigr).
\end{equation}

\subsection{The Boltzmann entropy}

Menon and Yu compute the volume of the balanced factorization fiber over a full-rank matrix \cite{MenonYu2026}.  Up to an additive constant depending only on $N$ and $d$, the resulting entropy is
\begin{equation}\label{eq:entropy-matrix}
S_N(W)=\Sn(\sigma(W)),
\end{equation}
where
\begin{equation*}
\Sn(\sigma_1,\ldots,\sigma_d)
=
\frac12\sum_{1\leq i<j\leq d}
\log
\frac{\sigma_i^2-\sigma_j^2}
{\sigma_i^{2/N}-\sigma_j^{2/N}}.
\end{equation*}
The expression extends real-analytically across positive repeated singular values \cite[Theorem 1.8]{ChenKotwalMenon2025}.  Since it is symmetric in the singular values, it defines a smooth orthogonally invariant function on $\GLd$.

The diagonal-completion energy and its Euclidean gradient are
\begin{equation*}
E_d(W)=\frac12\sum_{i=1}^d(W_{ii}-1)^2,
\qquad
E_d'(W)=\diag(W_{11}-1,\ldots,W_{dd}-1).
\end{equation*}
The free energy is
\begin{equation*}
F_\beta(W)=E_d(W)-\frac1\beta S_N(W).
\end{equation*}

\begin{remark}[Rank-deficient strata]\label{rem:rank-strata}
The entropy formula \eqref{eq:entropy-matrix} is the full-rank fiber entropy.  When we take limits toward a rank-deficient matrix, we study the boundary behavior of this full-rank function and of the full-rank flow.  We do not identify that boundary value with the volume of a fiber in a lower-rank balanced stratum, whose dimension and normalization differ.
\end{remark}
The following entropy result of Chen, Kotwal, and Menon will be used in the equilibrium analysis below.
\begin{theorem}[Strict concavity in singular-value coordinates \cite{ChenKotwalMenon2025}]\label{thm:strict-concavity}
If $N>2$ and $d\geq2$, then
\[
\nabla^2\Sn(\sigma)\prec0
\qquad
\text{for every }\sigma\in(0,\infty)^d.
\]
More generally, strict concavity holds for every $(N,d)\neq(2,2)$.
\end{theorem}

We also use the standard differential calculus of singular-value functions.  If $f$ is a smooth symmetric function on $(0,\infty)^d$ and
\[
\Phi(W)=f(\sigma(W)),
\]
then at an SVD $W=U\Sigma V^T$,
\begin{equation}\label{eq:spectral-gradient}
\Phi'(W)=U\diag(\partial_1f,\ldots,\partial_df)V^T.
\end{equation}
At repeated singular values, symmetry makes the right-hand side independent of the chosen singular basis.  We give the second-order formula needed later in \cref{sec:hessian}.  General treatments of this calculus include \cite{Lewis1996,LewisSendov2001,DrusvyatskiyKempton2015}.

\section{Pairwise entropy calculus}\label{sec:pairwise}

The entropy is a sum of two-variable terms.  It is useful to isolate them.

\begin{definition}\label{def:qN}
For $a,b>0$, define
\begin{equation}\label{eq:qN}
q_N(a,b)
=
\frac12\log
\frac{a^2-b^2}{a^{2/N}-b^{2/N}},
\end{equation}
with the analytic continuation at $a=b$.  Put
\[
\cN=1-\frac1N.
\]
Then
\[
\Sn(\sigma)=\sum_{i<j}q_N(\sigma_i,\sigma_j).
\]
\end{definition}

\begin{lemma}[Basic identities]\label{lem:q-basic}
For every integer $N\geq2$:
\begin{enumerate}[label=\textup{(\roman*)}]
\item $q_N$ is symmetric and real analytic on $(0,\infty)^2$, including the diagonal, and
\[
q_N(a,a)=\frac12\log N+\cN\log a.
\]

\item For $t>0$,
\begin{equation}\label{eq:q-homogeneity}
q_N(ta,tb)=q_N(a,b)+\cN\log t.
\end{equation}

\item The function $q_N$ is strictly increasing in each argument.

\item If $a\geq b>0$ and $r=b/a$, then
\begin{equation}\label{eq:q-ratio}
q_N(a,b)=\cN\log a+h_N(r),
\qquad
h_N(r)=\frac12\log\frac{1-r^2}{1-r^{2/N}}.
\end{equation}
The function $h_N$ extends continuously to $[0,1]$, with
\[
h_N(0)=0,
\qquad
h_N(1)=\frac12\log N.
\]
In particular, $h_N$ is bounded.
\end{enumerate}
\end{lemma}

\begin{proof}
Set $x=a^{2/N}$ and $y=b^{2/N}$.  The divided-difference identity
\begin{equation}\label{eq:geometric-sum}
\frac{a^2-b^2}{a^{2/N}-b^{2/N}}
=
\frac{x^N-y^N}{x-y}
=
\sum_{m=0}^{N-1}x^{N-1-m}y^m
\end{equation}
shows symmetry, positivity, analyticity across $x=y$, and strict increase in each positive variable.  At $a=b$, the sum is $Na^{2-2/N}$.  Homogeneity and \eqref{eq:q-ratio} follow by factoring out $a^{2-2/N}$.  The endpoint values of $h_N$ follow directly, using l'Hospital's rule at $r=1$.
\end{proof}

Write
\begin{equation*}
g_i(\sigma)=\partial_i\Sn(\sigma).
\end{equation*}

\begin{corollary}[Positivity of the entropy gradient]\label{cor:g-positive}
For $d\geq2$ and every $\sigma\in(0,\infty)^d$,
\[
g_i(\sigma)>0,
\qquad i=1,\ldots,d.
\]
\end{corollary}

\begin{proof}
Every $g_i$ is the sum of the partial derivatives with respect to the $i$th variable of the $d-1$ pair terms involving $\sigma_i$.  Each derivative is positive by \cref{lem:q-basic}.
\end{proof}

\begin{lemma}[Strict reverse ordering of the gradient]\label{lem:gradient-order}
Assume $N>2$.  If $\sigma_i\neq\sigma_j$, then
\begin{equation}\label{eq:gradient-order}
(\sigma_i-\sigma_j)(g_i(\sigma)-g_j(\sigma))<0.
\end{equation}
Consequently,
\[
g_i(\sigma)=g_j(\sigma)
\quad\Longleftrightarrow\quad
\sigma_i=\sigma_j.
\]
\end{lemma}

\begin{proof}
Let $\tau$ exchange coordinates $i$ and $j$, and put $\widetilde\sigma=\tau\sigma$.  Symmetry gives $\nabla\Sn(\widetilde\sigma)=\tau\nabla\Sn(\sigma)$.  By \cref{thm:strict-concavity},
\[
\bigl(\nabla\Sn(\sigma)-\nabla\Sn(\widetilde\sigma)\bigr)
\cdot(\sigma-\widetilde\sigma)<0
\]
when $\sigma\neq\widetilde\sigma$.  The left-hand side is
\[
2(\sigma_i-\sigma_j)(g_i-g_j).
\]
This proves \eqref{eq:gradient-order}.  The converse implication follows from symmetry when $\sigma_i=\sigma_j$.
\end{proof}

\begin{lemma}[Boundary derivative]\label{lem:boundary-derivative}
Fix $0<m<M<\infty$.  Uniformly for $a\in[m,M]$, as $b\downarrow0$,
\begin{equation*}
\partial_bq_N(a,b)
=
\frac1N a^{-2/N}b^{2/N-1}\bigl(1+o(1)\bigr).
\end{equation*}
Equivalently, with $p=2-2/N$,
\begin{equation}\label{eq:boundary-drift-pair}
N b^p\partial_bq_N(a,b)
=
a^{-2/N}b+o(b).
\end{equation}
The exact identity is
\begin{equation}\label{eq:boundary-drift-exact}
N b^p\partial_bq_N(a,b)
=
\frac{b}{a^{2/N}-b^{2/N}}
-
\frac{N b^{3-2/N}}{a^2-b^2}.
\end{equation}
\end{lemma}

\begin{proof}
Differentiate \eqref{eq:qN} with respect to the smaller argument:
\[
\partial_bq_N(a,b)
=
-\frac{b}{a^2-b^2}
+
\frac1N\frac{b^{2/N-1}}{a^{2/N}-b^{2/N}}.
\]
The second term is dominant when $N>2$, and its expansion is uniform for $a\in[m,M]$.  Multiplication by $Nb^p$ gives \eqref{eq:boundary-drift-exact} and \eqref{eq:boundary-drift-pair}.
\end{proof}

\begin{proposition}[Exterior-volume representation and logarithmic bounds]\label{lem:entropy-upper}
Order the singular values as $\sigma_1\geq\cdots\geq\sigma_d>0$.  Then
\begin{equation}\label{eq:entropy-weighted-singular-exact}
S_N(W)
=
\cN\sum_{i=1}^{d-1}(d-i)\log\sigma_i
+R_N(W),
\end{equation}
where
\begin{equation}\label{eq:R-bound}
0\leq R_N(W)\leq\frac12\binom d2\log N.
\end{equation}
Equivalently,
\begin{equation}\label{eq:entropy-exterior-exact}
S_N(W)
=
\cN\sum_{k=1}^{d-1}
\log\left\|\bigwedge\nolimits^kW\right\|_{\mathrm{op}}
+R_N(W).
\end{equation}
Consequently,
\begin{equation}\label{eq:entropy-exterior-upper}
S_N(W)
\leq
C_{N,d}
+
\cN\sum_{k=1}^{d-1}
\log^+\left\|\bigwedge\nolimits^kW\right\|_{\mathrm{op}},
\end{equation}
and
\begin{equation}\label{eq:entropy-upper}
S_N(W)
\leq
C_{N,d}
+
\binom d2\cN\log(1+\sigma_1(W)).
\end{equation}
\end{proposition}

\begin{proof}
For $i<j$, write $r_{ij}=\sigma_j/\sigma_i\in(0,1]$.  By \eqref{eq:q-ratio},
\[
q_N(\sigma_i,\sigma_j)
=
\cN\log\sigma_i+h_N(r_{ij}),
\]
where $0\leq h_N(r)\leq\frac12\log N$ by \eqref{eq:geometric-sum}.  Summing over $i<j$ gives \eqref{eq:entropy-weighted-singular-exact}, with
\[
R_N(W)=\sum_{i<j}h_N(r_{ij}),
\]
and proves \eqref{eq:R-bound}.

The singular values of the induced map $\bigwedge^kW$ are all products $\sigma_{i_1}\cdots\sigma_{i_k}$, so
\[
\left\|\bigwedge\nolimits^kW\right\|_{\mathrm{op}}
=\sigma_1\cdots\sigma_k.
\]
Therefore
\[
\sum_{k=1}^{d-1}
\log\left\|\bigwedge\nolimits^kW\right\|_{\mathrm{op}}
=
\sum_{k=1}^{d-1}\sum_{i=1}^k\log\sigma_i
=
\sum_{i=1}^{d-1}(d-i)\log\sigma_i,
\]
which proves \eqref{eq:entropy-exterior-exact}.  Replacing each logarithm by its positive part yields \eqref{eq:entropy-exterior-upper}.  Finally,
\[
\left\|\bigwedge\nolimits^kW\right\|_{\mathrm{op}}
\leq\sigma_1(W)^k,
\]
and $\sum_{k=1}^{d-1}k=\binom d2$, giving \eqref{eq:entropy-upper} after an adjustment of the constant for $\sigma_1\leq1$.
\end{proof}

\section{When is the free energy bounded below?}\label{sec:general-energy}

The diagonal-completion loss is one example of a broader obstruction.  In this section we separate three issues that are sometimes conflated:
\begin{enumerate}[label=\textup{(\roman*)}]
\item whether $F_\beta$ is bounded from below;
\item whether $F_\beta(W)\to+\infty$ as $\normF{W}\to\infty$;
\item whether the sublevel sets of $F_\beta$ are compact in the full-rank state space $\GLd$.
\end{enumerate}
The first property rules out arbitrarily negative values, the second gives confinement in norm, and the third is the natural condition for an equilibrium selection principle on the open manifold $\GLd$.  Boundedness from below alone does not prevent divergent gradient trajectories: the function $f(x)=e^{-x}$ is bounded below on $\mathbb R$, but its negative-gradient flow satisfies $\dot x=e^{-x}$ and has $x(t)\to+\infty$.

\subsection{The logarithmic entropy scale}

The exterior-volume formula in \cref{lem:entropy-upper} shows that the fiber entropy grows at most logarithmically.  It is convenient to write
\begin{equation}\label{eq:Psi-def}
\Psi(W)
=
\sum_{k=1}^{d-1}
\log^+\left\|\bigwedge\nolimits^k W\right\|_{\mathrm{op}},
\qquad
\log^+x=\max\{\log x,0\}.
\end{equation}
Then \cref{lem:entropy-upper} gives
\begin{equation}\label{eq:entropy-Psi-upper}
S_N(W)\leq C_{N,d}+\cN\Psi(W),
\end{equation}
and, more crudely,
\begin{equation}\label{eq:entropy-norm-upper-general}
S_N(W)
\leq
C_{N,d}+\alpha_{N,d}\log(1+\normF{W}),
\qquad
\alpha_{N,d}=\cN\binom d2.
\end{equation}
The coefficient in the norm estimate is sharp at the level of a uniform radial bound: along $W=tI_d$,
\begin{equation*}
S_N(tI_d)=\alpha_{N,d}\log t+C_{N,d}.
\end{equation*}
Thus polynomial growth of the energy is far more than is needed to dominate the entropy.  The natural threshold is logarithmic.

\subsection{An exact criterion from energy sublevel sets}

Let $E:\GLd\to\mathbb R$ be bounded below.  Adding a constant to $E$ changes $F_\beta$ by the same constant and has no effect on boundedness or on the gradient flow, so we normalize $E\geq0$.  Define the \emph{entropy profile of the energy} by
\begin{equation*}
\Theta_E(r)
=
\sup\{S_N(W):W\in\GLd,\ E(W)\leq r\},
\qquad r\geq0,
\end{equation*}
with values in $\mathbb R\cup\{+\infty\}$.

\begin{proposition}[Exact sublevel criterion]\label{prop:Theta-criterion}
For every $\beta>0$,
\begin{equation}\label{eq:Theta-duality}
\sup_{W\in\GLd}\bigl(S_N(W)-\beta E(W)\bigr)
=
\sup_{r\geq0}\bigl(\Theta_E(r)-\beta r\bigr)
\end{equation}
in the extended real numbers.  Consequently,
\begin{equation}\label{eq:Theta-bounded-iff}
F_\beta\text{ is bounded below}
\quad\Longleftrightarrow\quad
\sup_{r\geq0}\bigl(\Theta_E(r)-\beta r\bigr)<\infty.
\end{equation}
\end{proposition}

\begin{proof}
For any $W\in\GLd$, setting $r=E(W)$ gives
\[
S_N(W)-\beta E(W)
\leq
\Theta_E(r)-\beta r,
\]
so the left-hand side of \eqref{eq:Theta-duality} is at most the right-hand side.  Conversely, if $E(W)\leq r$, then
\[
S_N(W)-\beta r
\leq
S_N(W)-\beta E(W).
\]
Taking the supremum over such $W$ gives
\[
\Theta_E(r)-\beta r
\leq
\sup_{W\in\GLd}\bigl(S_N(W)-\beta E(W)\bigr).
\]
Now take the supremum over $r$.  Since
\[
\inf F_\beta
=-\frac1\beta
\sup_W\bigl(S_N(W)-\beta E(W)\bigr),
\]
\eqref{eq:Theta-bounded-iff} follows.
\end{proof}

\begin{corollary}\label{cor:Theta-consequences}
The following statements hold.
\begin{enumerate}[label=\textup{(\roman*)}]
\item If $\Theta_E(r_0)=+\infty$ for some finite $r_0$, then $F_\beta$ is unbounded below for every finite $\beta>0$.
\item If $\Theta_E(r)\leq ar+C$ for all $r\geq0$, then $F_\beta$ is bounded below for every $\beta\geq a$.
\item If $\Theta_E(r)<\infty$ for finite $r$ and $\Theta_E(r)=o(r)$ as $r\to\infty$, then $F_\beta$ is bounded below for every $\beta>0$.
\end{enumerate}
\end{corollary}

\begin{proof}
Each statement follows immediately from \eqref{eq:Theta-bounded-iff}.  For part \textup{(iii)}, choose $R$ so that $\Theta_E(r)\leq \beta r/2$ for $r\geq R$ and use finiteness of $\Theta_E$ on the compact interval $[0,R]$.
\end{proof}

The criterion pinpoints the obstruction in diagonal completion: there is unlimited entropy already at zero loss.  In the notation above,
\begin{equation*}
\Theta_{E_d}(0)=+\infty.
\end{equation*}

\subsection{Affine flat directions}

The preceding failure is structural for any loss with a nontrivial affine direction on which the energy stays bounded.

\begin{lemma}[Entropy diverges on every nontrivial affine ray]\label{lem:affine-ray}
Let $W_0\in\GLd$ and $0\neq H\in\Md$.  Then $W_0+tH$ is invertible for all but finitely many $t\in\mathbb R$, and
\begin{equation}\label{eq:affine-ray-entropy}
S_N(W_0+tH)\longrightarrow+\infty
\qquad\text{as }|t|\to\infty
\end{equation}
through those values for which $W_0+tH\in\GLd$.
\end{lemma}

\begin{proof}
The function $t\mapsto\det(W_0+tH)$ is a polynomial that is nonzero at $t=0$, so it is not identically zero and has only finitely many real roots.

For $1\leq k\leq d-1$, the matrix of $\bigwedge^k(W_0+tH)$ in the standard exterior basis consists of the $k\times k$ minors of $W_0+tH$, each of which is a polynomial in $t$.  Since $W_0$ is invertible, at least one $k\times k$ minor is nonzero at $t=0$.  The corresponding polynomial is not identically zero and hence, for all sufficiently large $|t|$, its absolute value is bounded below by a positive constant.  Therefore
\begin{equation}\label{eq:wedge-lower-affine}
\left\|\bigwedge\nolimits^k(W_0+tH)\right\|_{\mathrm{op}}
\geq c_k>0
\end{equation}
for large $|t|$.

For $k=1$, choose an entry $H_{ij}\neq0$.  Then
\[
|(W_0+tH)_{ij}|\longrightarrow\infty,
\]
so $\|W_0+tH\|_{\mathrm{op}}\to\infty$.  The exterior-volume representation \eqref{eq:entropy-exterior-exact} now has one term tending to $+\infty$, while all remaining terms are bounded below by \eqref{eq:wedge-lower-affine}; its remainder is uniformly bounded.  This proves \eqref{eq:affine-ray-entropy}.
\end{proof}

\begin{corollary}[Affine-ray obstruction]\label{cor:affine-energy-obstruction}
Suppose that for some $W_0\in\GLd$ and $0\neq H\in\Md$,
\begin{equation*}
\sup_{|t|\geq t_0}E(W_0+tH)<\infty.
\end{equation*}
Then $F_\beta$ is unbounded below for every finite $\beta>0$.  In particular, this holds if $E$ is constant on a nontrivial affine ray through a full-rank matrix.
\end{corollary}

\begin{proof}
Combine \eqref{eq:affine-ray-entropy} with the assumed upper bound for the energy.
\end{proof}

\subsection{A dichotomy for linear matrix sensing}

Let $\mathcal L:\Md\to\mathbb R^m$ be linear, let $y\in\mathbb R^m$, and consider the squared sensing loss
\begin{equation*}
E_{\mathcal L,y}(W)
=
\frac12\|\mathcal L(W)-y\|_2^2.
\end{equation*}

\begin{theorem}[Linear-sensing dichotomy]\label{thm:linear-sensing-dichotomy}
For every $d\geq2$, $N\geq2$, and $\beta>0$, the free energy
\[
F_{\beta,\mathcal L,y}(W)
=
E_{\mathcal L,y}(W)-\frac1\beta S_N(W)
\]
is bounded below on $\GLd$ if and only if $\mathcal L$ is injective.  If $\mathcal L$ is injective, then
\begin{equation*}
F_{\beta,\mathcal L,y}(W)\longrightarrow+\infty
\qquad\text{as }\normF{W}\to\infty.
\end{equation*}
\end{theorem}

\begin{proof}
Suppose first that $\mathcal L$ is not injective.  Choose $0\neq H\in\ker\mathcal L$ and any $W_0\in\GLd$.  Then
\[
E_{\mathcal L,y}(W_0+tH)
=
E_{\mathcal L,y}(W_0)
\]
for every $t$.  The conclusion follows from \cref{cor:affine-energy-obstruction}.

Now suppose that $\mathcal L$ is injective.  By finite dimensionality, there is $s_{\mathcal L}>0$ such that
\begin{equation*}
\|\mathcal L(W)\|_2\geq s_{\mathcal L}\normF{W}
\qquad\text{for all }W\in\Md.
\end{equation*}
The elementary inequality $\|u-v\|_2^2\geq\frac12\|u\|_2^2-\|v\|_2^2$ gives
\begin{equation*}
E_{\mathcal L,y}(W)
\geq
\frac{s_{\mathcal L}^2}{4}\normF{W}^2
-
\frac12\|y\|_2^2.
\end{equation*}
Combining this with \eqref{eq:entropy-norm-upper-general},
\begin{equation*}
F_{\beta,\mathcal L,y}(W)
\geq
\frac{s_{\mathcal L}^2}{4}\normF{W}^2
-
\frac{\alpha_{N,d}}{\beta}\log(1+\normF{W})
-C,
\end{equation*}
which tends to $+\infty$ with $\normF{W}$.  The same inequality gives a global lower bound.
\end{proof}

\begin{corollary}[Underdetermined matrix sensing]\label{cor:underdetermined-sensing}
On the unrestricted matrix space, every genuinely underdetermined linear sensing loss has unbounded DLN free energy.  In particular, if $P_\Omega$ observes a proper subset $\Omega$ of the $d^2$ entries, then
\[
\frac12\|P_\Omega(W)-y\|_F^2-\frac1\beta S_N(W)
\]
is unbounded below for every finite $\beta>0$.
\end{corollary}

\begin{remark}
If the admissible matrices are restricted to a linear subspace $V\subset\Md$, injectivity in \cref{thm:linear-sensing-dichotomy} should be understood for $\mathcal L|_V$.  The affine-ray argument applies whenever $V\cap\GLd$ contains a point and $\ker(\mathcal L|_V)$ contains a nonzero direction through which one can remain in $V\cap\GLd$ for arbitrarily large parameter values.
\end{remark}

\subsection{Practical sufficient conditions}

The entropy estimate gives conditions that apply beyond quadratic sensing losses.

\begin{proposition}[Exterior-volume domination]\label{prop:exterior-domination}
Let $E:\GLd\to\mathbb R$ be bounded below and let $\Psi$ be defined by \eqref{eq:Psi-def}.
\begin{enumerate}[label=\textup{(\roman*)}]
\item If
\begin{equation*}
E(W)
\geq
\frac{\cN}{\beta}\Psi(W)-C,
\end{equation*}
then $F_\beta$ is bounded below.
\item If there is $\delta>0$ such that
\begin{equation}\label{eq:exterior-coercive-condition}
E(W)
\geq
\left(\frac{\cN}{\beta}+\delta\right)\Psi(W)-C,
\end{equation}
then $F_\beta(W)\to+\infty$ as $\normF{W}\to\infty$.
\end{enumerate}
\end{proposition}

\begin{proof}
The first claim follows immediately from \eqref{eq:entropy-Psi-upper}.  Under \eqref{eq:exterior-coercive-condition},
\[
F_\beta(W)\geq\delta\Psi(W)-C'.
\]
Since the $k=1$ summand in $\Psi$ is $\log^+\|W\|_{\mathrm{op}}$, the right-hand side tends to $+\infty$ when $\normF{W}\to\infty$.
\end{proof}

\begin{corollary}[Norm-growth criteria]\label{cor:norm-growth-criteria}
Let $\alpha_{N,d}=\cN\binom d2$.
\begin{enumerate}[label=\textup{(\roman*)}]
\item If
\begin{equation*}
E(W)
\geq
\frac{\alpha_{N,d}}{\beta}\log(1+\normF{W})-C,
\end{equation*}
then $F_\beta$ is bounded below.
\item If for some $\delta>0$,
\begin{equation*}
E(W)
\geq
\left(\frac{\alpha_{N,d}}{\beta}+\delta\right)
\log(1+\normF{W})-C,
\end{equation*}
then $F_\beta$ is coercive in norm.
\item If
\begin{equation}\label{eq:superlog-condition}
\frac{E(W)}{\log(1+\normF{W})}
\longrightarrow+\infty
\qquad\text{as }\normF{W}\to\infty,
\end{equation}
then $F_\beta$ is coercive in norm for every finite $\beta>0$.
\item In particular, any estimate
\begin{equation*}
E(W)\geq a\normF{W}^p-C,
\qquad a>0,
\quad p>0,
\end{equation*}
implies norm coercivity of $F_\beta$ for every finite $\beta>0$.
\end{enumerate}
\end{corollary}

\begin{proof}
Use \eqref{eq:entropy-norm-upper-general}.  Parts \textup{(iii)} and \textup{(iv)} imply the strict logarithmic domination in part \textup{(ii)} outside a sufficiently large ball.
\end{proof}

The adjective ``superlogarithmic'' in \eqref{eq:superlog-condition} is essential.  Ordinary coercivity of $E$ is not enough: a coercive energy that grows more slowly than $\log\normF{W}$ can still be dominated by $S_N$ along scalar rays.

A sublevel-set version is sometimes more convenient.  Define
\begin{equation*}
R_E(r)
=
\sup\{\normF{W}:W\in\GLd,\ E(W)\leq r\}
\in[0,+\infty].
\end{equation*}
Whenever $R_E(r)<\infty$, \eqref{eq:entropy-norm-upper-general} implies
\begin{equation*}
\Theta_E(r)
\leq
C_{N,d}+\alpha_{N,d}\log(1+R_E(r)).
\end{equation*}
Thus the condition
\begin{equation*}
\sup_{r\geq0}
\left[
\alpha_{N,d}\log(1+R_E(r))-\beta r
\right]<\infty
\end{equation*}
is sufficient for boundedness below.  In particular,
\begin{equation*}
\log(1+R_E(r))=o(r)
\end{equation*}
is sufficient for every $\beta>0$.

\subsection{Boundedness, coercivity, and attainment}

Even norm coercivity does not automatically yield a minimizer on $\GLd$, because $\GLd$ is open and a minimizing sequence may approach the rank boundary.  The boundary behavior of the entropy explains exactly where an additional hypothesis is needed.

\begin{proposition}[Entropy at the finite rank boundary]\label{prop:entropy-boundary-general}
Let $W_m\in\GLd$ be bounded and suppose $W_m\to W_*$.
\begin{enumerate}[label=\textup{(\roman*)}]
\item If $\rank W_*\leq d-2$, then $S_N(W_m)\to-\infty$.
\item If $\rank W_*=d-1$, then $S_N(W_m)$ has a finite limit determined by the $d-1$ positive limiting singular values.  More precisely, if these are $\tau_1,\ldots,\tau_{d-1}>0$, then
\begin{align}\label{eq:corank-one-entropy-limit}
\lim_{m\to\infty}S_N(W_m)
&=
\sum_{1\leq i<j\leq d-1}q_N(\tau_i,\tau_j)
+
\cN\sum_{i=1}^{d-1}\log\tau_i.
\end{align}
\end{enumerate}
\end{proposition}

\begin{proof}
If at least two singular values tend to zero, the pair term involving those two singular values tends to $-\infty$ by the homogeneity relation \eqref{eq:q-homogeneity}; all pair terms are bounded above on a bounded set by \cref{lem:entropy-upper}.  This proves part \textup{(i)}.  If only $\sigma_d\to0$, then for each $i<d$, \eqref{eq:q-ratio} and $h_N(0)=0$ give
\[
q_N(\sigma_i,\sigma_d)\longrightarrow\cN\log\tau_i,
\]
while all pairs among the first $d-1$ singular values converge continuously.  Summing proves \eqref{eq:corank-one-entropy-limit}.
\end{proof}

\begin{corollary}[A sufficient condition for attainment]\label{cor:attainment-condition}
Assume that $E:\GLd\to\mathbb R$ is continuous and bounded below, that $F_\beta$ is coercive in norm, and that
\begin{equation}\label{eq:energy-corank-one-barrier}
E(W_m)\longrightarrow+\infty
\end{equation}
for every bounded sequence $W_m\in\GLd$ converging to a matrix of rank $d-1$.  Then every sublevel set of $F_\beta$ is compact in $\GLd$, and $F_\beta$ attains a global minimum.
\end{corollary}

\begin{proof}
Norm coercivity bounds every sublevel set in $\Md$.  A sequence in such a set cannot approach rank at most $d-2$, because \cref{prop:entropy-boundary-general} and the lower bound on $E$ would force $F_\beta\to+\infty$.  It cannot approach rank $d-1$ because of \eqref{eq:energy-corank-one-barrier} and the finite entropy limit in \eqref{eq:corank-one-entropy-limit}.  Hence the sublevel set is closed in a bounded subset of $\Md$ and stays a positive distance from the rank boundary.  It is compact in $\GLd$, and the direct method gives a minimizer.
\end{proof}

For example, ridge-regularized completion,
\begin{equation*}
E_{\Omega,\lambda}(W)
=
\frac12\|P_\Omega(W)-Y\|_F^2
+
\frac\lambda2\normF{W}^2,
\qquad\lambda>0,
\end{equation*}
is norm-coercive and therefore makes $F_\beta$ bounded below.  The ridge term does not create a corank-one barrier, however, so additional treatment of lower-rank strata or a barrier such as $-\mu\log\det(W^TW)$ is still needed if one insists on a minimizer inside $\GLd$.

\subsection{When entropy does define a selection principle}

The usual low-temperature selection statement becomes valid once compactness and attainment are supplied.

\begin{proposition}[Low-temperature entropy selection]\label{prop:low-temp-selection}
Let $\mathcal X$ be a closed subset of a finite-dimensional Euclidean space, and let $E,S:\mathcal X\to\mathbb R$ be continuous.  Put
\[
m=\min_{\mathcal X}E,
\qquad
\mathcal M=\operatorname*{argmin}_{\mathcal X}E,
\]
and assume $\mathcal M\neq\varnothing$.  Suppose that for every sufficiently large $\beta$, the function
\[
F_\beta=E-\beta^{-1}S
\]
has a minimizer $W_\beta$, and that the family $\{W_\beta\}$ is precompact.  Then every limit point $W_*$ of $W_\beta$ as $\beta\to\infty$ satisfies
\begin{equation}\label{eq:low-temp-selection-conclusion}
W_*\in\operatorname*{argmax}_{W\in\mathcal M}S(W).
\end{equation}
If the entropy has a unique maximizer on $\mathcal M$, then $W_\beta$ converges to it.
\end{proposition}

\begin{proof}
Fix $V\in\mathcal M$.  Minimality of $W_\beta$ gives
\[
E(W_\beta)-\frac1\beta S(W_\beta)
\leq
m-\frac1\beta S(V),
\]
so
\begin{equation}\label{eq:selection-comparison}
0\leq E(W_\beta)-m
\leq
\frac{S(W_\beta)-S(V)}{\beta}.
\end{equation}
Precompactness and continuity make $S(W_\beta)$ bounded along every sequence $\beta\to\infty$, so \eqref{eq:selection-comparison} implies $E(W_\beta)\to m$.  Every limit point therefore belongs to $\mathcal M$.  The same inequality also gives
\[
S(W_\beta)\geq S(V)
\qquad\text{for every }V\in\mathcal M.
\]
Passing to a convergent subsequence yields $S(W_*)\geq S(V)$ for every $V\in\mathcal M$, proving \eqref{eq:low-temp-selection-conclusion}.  Uniqueness of the entropy maximizer then forces the whole precompact family to converge.
\end{proof}

Thus a natural setting for entropic selection is a compact energy-minimizing set on which the entropy is continuous and has a distinguished maximum.  Diagonal completion violates this picture in the strongest possible way: its minimizing set is noncompact and its entropy is unbounded above there.  Instead of choosing a finite minimizer, the entropic term drives the system toward larger and larger unobserved directions.

\subsection{Applied energies: how to read the examples}
\label{subsec:applied-examples}

The preceding results separate three questions that should be kept distinct in applications.  First, does the energy $E$ attain its infimum on the natural closed parameter space?  Second, is
\[
F_\beta(W)=E(W)-\frac{1}{\beta}S_N(W)
\]
bounded below, or even coercive in norm?  Third, if the state space is the open manifold $\GLd$, is the infimum of $F_\beta$ attained away from the rank boundary?  An energy may have many minimizers while $F_\beta$ is unbounded below, as happens for underdetermined least squares.  Conversely, $F_\beta$ may be norm-coercive even when the minimizer of $E$ lies on the excluded rank boundary.  The examples below record the energy, its elementary variational properties, and the conclusion supplied by the criteria in this section.  No proofs are repeated.

The universal estimate
\[
S_N(W)\leq C_{N,d}+\alpha_{N,d}\log(1+\normF{W}),
\qquad
\alpha_{N,d}=\left(1-\frac1N\right)\binom d2,
\]
should be kept in mind throughout.  Positive polynomial growth of $E$ dominates the entropy at infinity, whereas any nontrivial affine direction on which $E$ stays bounded forces $F_\beta$ to be unbounded below.  Logarithmically growing energies lie at the threshold and must be tested direction by direction.  Finally, norm coercivity does not by itself settle attainment in $\GLd$, because the entropy has a finite limit at a generic corank-one boundary point; see \cref{prop:entropy-boundary-general,cor:attainment-condition}.

\subsection{Quadratic losses, inverse problems, and low-rank regularization}
\label{subsec:quadratic-applied-energies}

\subsubsection{Full quadratic loss and Gaussian matrix denoising}

The basic full-observation energy is
\begin{equation*}
E_A(W)=\frac{1}{2\sigma^2}\normF{W-A}^2,
\qquad A\in\Md,
\end{equation*}
where $\sigma^2$ may be interpreted as a noise variance.  This is the negative log-likelihood, up to an additive constant, for observing $A$ in isotropic Gaussian matrix noise.  Such losses also underlie singular-value shrinkage and matrix denoising \cite{GavishDonoho2017}.

\emph{Properties of the energy.}
On the closed space $\Md$, the unique minimizer is $A$, and the energy has quadratic growth.  If $A\in\GLd$, the energy minimum is attained in the full-rank state space.  If $A$ is singular, the infimum of $E_A|_{\GLd}$ equals zero but is not attained.

\emph{Conclusion for the free energy.}
By \cref{cor:norm-growth-criteria}, $F_\beta$ is bounded below and norm-coercive for every depth and every finite $\beta>0$.  If $A$ is singular, or more generally if a minimizing sequence can approach corank one, norm coercivity alone does not decide whether the free-energy infimum is attained in $\GLd$.  An explicit rank-boundary barrier gives attainment by \cref{cor:attainment-condition}.

\subsubsection{Injective squared sensing and identifiable least squares}

Let $\mathcal L:\Md\to\mathbb R^m$ be linear and consider
\begin{equation}\label{eq:applied-injective-sensing}
E_{\mathcal L,y}(W)
=
\frac12\|\mathcal L(W)-y\|_2^2.
\end{equation}
This form includes multivariate linear regression, linear system identification, and matrix sensing.  Low-rank matrix-sensing theory often assumes a restricted isometry on low-rank matrices rather than injectivity on all of $\Md$ \cite{CandesPlan2011}.

\emph{Properties of the energy.}
If $\mathcal L$ is injective on the actual state space, then $\mathcal L^*\mathcal L$ is positive definite there.  The least-squares energy has a unique minimizer and satisfies a quadratic lower bound
\[
E_{\mathcal L,y}(W)\geq c\normF{W}^2-C.
\]
If the state space is a proper linear model $V\subset\Md$, the relevant hypothesis is injectivity of $\mathcal L|_V$.

\emph{Conclusion for the free energy.}
The linear-sensing dichotomy, \cref{thm:linear-sensing-dichotomy}, shows that $F_\beta$ is norm-coercive and bounded below for every $N$ and $\beta$.  As before, attainment inside $\GLd$ requires either a separate corank-one argument, a closed stratified state space, or a boundary barrier.

\subsubsection{Underdetermined least squares and matrix completion}

The same energy \eqref{eq:applied-injective-sensing} behaves completely differently when $\ker\mathcal L\neq\{0\}$.  Entrywise matrix completion is the special case
\begin{equation*}
E_\Omega(W)
=
\frac12\|P_\Omega(W-Y)\|_F^2,
\end{equation*}
where $P_\Omega$ retains the observed entries.  Matrix completion and underdetermined matrix sensing are central models in low-rank recovery \cite{CandesRecht2009,CandesPlan2011,RechtFazelParrilo2010}.

\emph{Properties of the energy.}
A nonzero $H\in\ker\mathcal L$ gives
\[
E_{\mathcal L,y}(W_0+tH)=E_{\mathcal L,y}(W_0)
\]
for every $t$.  The ordinary energy may nevertheless attain its minimum, often on a noncompact affine set.  For incomplete entry observations on unrestricted matrices, $P_\Omega$ is noninjective whenever at least one entry is unobserved.

\emph{Conclusion for the free energy.}
By \cref{cor:affine-energy-obstruction,thm:linear-sensing-dichotomy}, $F_\beta$ is unbounded below for every finite $\beta>0$.  Thus the existence of exact or least-squares minimizers of $E$ does not produce an entropic selection principle.  The conclusion can change after restricting to a smaller model class on which the observation map is injective, or after introducing a confining penalty or proper prior.

\subsubsection{Ridge and Tikhonov regularization}

A standard confining modification is
\begin{equation*}
E_{\lambda}(W)
=
\frac12\|\mathcal L(W)-y\|_2^2
+
\frac{\lambda}{2}\normF{W}^2,
\qquad \lambda>0.
\end{equation*}
This is ridge or Tikhonov regularization, and in a Bayesian interpretation it corresponds to a centered Gaussian prior.  Ridge regression was introduced in \cite{HoerlKennard1970}; modern generalized-linear-model implementations are discussed in \cite{FriedmanHastieTibshirani2010}.

\emph{Properties of the energy.}
The energy is strongly convex on $\Md$, has a unique minimizer, and satisfies
\[
E_\lambda(W)\geq\frac{\lambda}{2}\normF{W}^2.
\]
This remains true whether or not $\mathcal L$ is injective.

\emph{Conclusion for the free energy.}
The free energy is bounded below and norm-coercive for every $N$ and $\beta$.  Ridge regularization therefore repairs the escape-to-infinity defect of matrix completion.  It does not itself diverge at rank loss, so attainment in $\GLd$ is still a separate boundary question.  Moreover, the ridge term supplies its own explicit selection bias, so any resulting selection is not attributable to entropy alone.

\subsubsection{Nuclear-norm regularization}

A common convex surrogate for low rank is
\begin{equation*}
E_{*,\lambda}(W)
=
\frac12\|\mathcal L(W)-y\|_2^2
+
\lambda\|W\|_*,
\qquad
\|W\|_*=\sum_{i=1}^d\sigma_i(W).
\end{equation*}
Nuclear-norm penalties are used throughout matrix completion, trace regression, and low-rank recovery \cite{RechtFazelParrilo2010,KoltchinskiiLouniciTsybakov2011}.

\emph{Properties of the energy.}
The energy is convex and coercive on $\Md$ for $\lambda>0$, hence it attains a minimum.  Its minimizers are frequently rank deficient, which is the intended regularizing effect.  The nuclear norm has linear rather than quadratic growth.

\emph{Conclusion for the free energy.}
Since $\|W\|_*\geq\normF{W}$, the free energy is norm-coercive for every $N$ and $\beta$.  The rank boundary remains important: both $\|W\|_*$ and $S_N$ have finite limits at a generic corank-one point.  If a zero singular value is opened by an amount $s$, the nuclear penalty changes at order $s$, while for $N>2$ the leading entropy gain is of order $s^{2/N}$ and for $N=2$ it is also linear.  Thus depth two is a coefficient-sensitive borderline case, and a global attainment statement requires a separate rank-boundary analysis.

\subsubsection{Layerwise quadratic weight decay and factor-induced Schatten penalties}

For a depth-$N$ factorization $W=W_N\cdots W_1$, the usual factor-space penalty is
\begin{equation*}
\mathcal R(W_1,\ldots,W_N)
=
\frac{\lambda}{2}\sum_{j=1}^N\normF{W_j}^2.
\end{equation*}
When the intermediate widths are large enough, minimizing this penalty over factorizations with fixed product gives
\begin{equation*}
\inf_{W_N\cdots W_1=W}\mathcal R(W_1,\ldots,W_N)
=
\frac{\lambda N}{2}\sum_{i=1}^d\sigma_i(W)^{2/N}.
\end{equation*}
The variational relation between factor penalties and Schatten-type penalties is used throughout deep matrix factorization and low-rank recovery \cite{GunasekarEtAl2017,AroraCohenHuLuo2019,GiampourasEtAl2020}; for two interacting matrices, the connection between weight decay, nuclear-norm regularization, and low-rank attention products is studied in \cite{KobayashiAkramVonOswald2024}.

\emph{Properties of the energy.}
The induced end-to-end penalty has positive-power growth at infinity, so it is confining.  For $N>2$ it is nonconvex and non-Lipschitz at a zero singular value.  Its small-singular-value scale is exactly $s^{2/N}$.

\emph{Conclusion for the free energy.}
The positive-power growth dominates the logarithmic entropy at infinity, so the effective free energy is norm-coercive.  Near a corank-one boundary point, however, the induced penalty and the entropy change on the same $s^{2/N}$ scale.  Consequently, whether a minimizer remains on the boundary or moves into $\GLd$ can depend on $\lambda$, $\beta$, the nonzero singular values, and the depth.  Factor-space weight decay therefore guarantees confinement but does not automatically settle full-rank attainment.

\subsection{Classification and margin-based energies}\label{subsec:classification-energies}

\subsubsection{Separable logistic and exponential losses}

For matrix-valued features $X_1,\ldots,X_M$ and labels $y_j\in\{-1,1\}$, consider
\begin{align}\label{eq:applied-logistic}
E_{\log}(W)
&=
\sum_{j=1}^M
\log\left(1+\exp\bigl(-y_j\langle W,X_j\rangle_F\bigr)\right),\\
E_{\exp}(W)
&=
\sum_{j=1}^M
\exp\bigl(-y_j\langle W,X_j\rangle_F\bigr).\notag
\end{align}
The existence of finite logistic maximum-likelihood estimators is governed by separation and overlap conditions \cite{AlbertAnderson1984}; the diverging-gradient-descent regime on separable data is analyzed in \cite{SoudryEtAl2018}.

\emph{Properties of the energy.}
If there is a separating matrix $H$ satisfying
\[
y_j\langle H,X_j\rangle_F>0
\qquad\text{for every }j,
\]
then $E_{\log}(W_0+tH)\to0$ and $E_{\exp}(W_0+tH)\to0$ as $t\to\infty$.  Their infimum is generally not attained at a finite parameter.

\emph{Conclusion for the free energy.}
The energy stays bounded along a nontrivial affine ray, so \cref{cor:affine-energy-obstruction} gives
\[
\inf_{\GLd}F_\beta=-\infty
\]
for every finite $\beta$.  The entropy does not cure the classical optimum-at-infinity phenomenon; it reinforces the failure of confinement.

\subsubsection{Identifiable nonseparable logistic loss}

The energy is again $E_{\log}$ in \eqref{eq:applied-logistic}, but the data are assumed to satisfy the usual no-separation and no-quasi-separation conditions, and the feature map is identifiable.

\emph{Properties of the energy.}
Under these hypotheses, logistic regression has a finite maximum-likelihood estimator \cite{AlbertAnderson1984}.  If the design spans the parameter space and there is positive recession in every nonzero direction, compactness of the unit sphere gives a linear lower bound
\[
E_{\log}(W)\geq c\normF{W}-C.
\]
Gauge symmetries or unobserved directions must first be removed; otherwise the energy is constant on affine fibers despite the absence of statistical separation.

\emph{Conclusion for the free energy.}
Linear growth is superlogarithmic, so $F_\beta$ is bounded below and norm-coercive for every finite $\beta$.  If the parametrization retains a nontrivial gauge direction, the affine-ray obstruction applies instead.  Thus identifiability is as important here as nonseparability.

\subsubsection{Multiclass softmax loss and gauge invariance}

For class-weight rows $w_1,\ldots,w_C$ and examples $(x_j,y_j)$, the multiclass cross-entropy is
\begin{equation*}
E_{\mathrm{soft}}(W)
=
-\sum_{j=1}^M
\log
\frac{\exp(w_{y_j}^{T}x_j)}{\sum_{c=1}^C\exp(w_c^Tx_j)}.
\end{equation*}
Multinomial logistic models and their regularization are discussed in \cite{FriedmanHastieTibshirani2010}.

\emph{Properties of the energy.}
For every $a$, adding the same score $a^Tx$ to each class leaves all softmax probabilities unchanged.  In matrix notation,
\[
E_{\mathrm{soft}}(W+\mathbf 1a^T)=E_{\mathrm{soft}}(W).
\]
This is an exact affine gauge symmetry, independent of whether the data are separable.

\emph{Conclusion for the free energy.}
On an unrestricted noncompact parametrization, the gauge direction forces $F_\beta$ to be unbounded below.  One must fix a reference class, impose a sum-to-zero gauge, quotient by the symmetry, or add a proper confining prior before applying an entropic selection argument.  

\subsubsection{Separable unregularized hinge loss}

The empirical hinge energy is
\begin{equation*}
E_{\mathrm{hinge}}(W)
=
\sum_{j=1}^M
\max\{0,1-y_j\langle W,X_j\rangle_F\}.
\end{equation*}
It is the basic loss in support-vector classification \cite{CortesVapnik1995}.

\emph{Properties of the energy.}
For separable data, $E_{\mathrm{hinge}}$ attains its minimum value zero, but it does so on an unbounded cone: once all margins exceed one, scaling the separator preserves zero loss.

\emph{Conclusion for the free energy.}
The free energy is unbounded below.  This is a useful contrast with separable logistic regression: hinge loss has genuine finite minimizers, but the minimizing set is noncompact and supports unbounded entropy.  Existence of minimizers for $E$ is therefore not enough for entropic selection.

\subsubsection{Quadratically regularized hinge loss}

The soft-margin support-vector energy is
\begin{equation*}
E_{\mathrm{SVM},\lambda}(W)
=
E_{\mathrm{hinge}}(W)
+
\frac{\lambda}{2}\normF{W}^2,
\qquad \lambda>0.
\end{equation*}
Quadratic margin regularization is part of the classical support-vector formulation \cite{CortesVapnik1995}.

\emph{Properties of the energy.}
The energy is strongly coercive and, on the ambient Euclidean space, has a unique minimizer whenever the quadratic term acts on every parameter direction.  It is nonsmooth only on the margin hyperplanes.

\emph{Conclusion for the free energy.}
Quadratic growth makes $F_\beta$ norm-coercive for every $N$ and $\beta$.  If the minimization is required to occur in $\GLd$, nonsmooth behavior at the rank boundary still deserves separate treatment, but entropy cannot cause escape to infinity.

\subsection{Robust-statistics energies}\label{subsec:robust-energies}

\subsubsection{Huber regression}

Let $r_j(W)$ be affine residuals and define
\begin{equation*}
E_{\mathrm H}(W)=\sum_{j=1}^M\rho_\delta(r_j(W)),
\qquad
\rho_\delta(u)
=
\begin{cases}
\frac12u^2,&|u|\leq\delta,\\
\delta|u|-\frac12\delta^2,&|u|>\delta.
\end{cases}
\end{equation*}
Huber's loss is a foundational robust-estimation criterion \cite{Huber1964}.

\emph{Properties of the energy.}
The loss is quadratic near zero and linear in the tails.  If the residual operator is injective on the state space, then
\[
E_{\mathrm H}(W)\geq c\normF{W}-C,
\]
so $E_{\mathrm H}$ is coercive and attains a minimum.  If the residual operator has a kernel, the energy is constant on the corresponding affine directions.

\emph{Conclusion for the free energy.}
In the identifiable case, linear growth dominates the logarithmic entropy and $F_\beta$ is norm-coercive.  In the nonidentifiable case, the affine-ray obstruction makes $F_\beta$ unbounded below.  Thus replacing squared loss by Huber loss changes the growth from quadratic to linear but does not change the decisive role of identifiability.

\subsubsection{Cauchy and Student-$t$ negative log-likelihoods}

A representative heavy-tailed energy is
\begin{equation*}
E_{\nu,s}(W)
=
\frac{\nu+1}{2}
\sum_{j=1}^M
\log\left(1+\frac{r_j(W)^2}{\nu s^2}\right),
\end{equation*}
with the Cauchy model obtained at $\nu=1$.  Student-$t$ likelihoods are widely used for robust statistical modeling \cite{LangeLittleTaylor1989}.

\emph{Properties of the energy.}
The loss is generally nonconvex in nonlinear models and has logarithmic tails.  Along a ray $W_0+tH$, each residual that grows linearly contributes a multiple of $\log|t|$, while residuals annihilated by $H$ contribute only a bounded amount.  If every residual is unchanged in some nonzero direction, the energy is exactly flat there.

\emph{Conclusion for the free energy.}
This is the borderline regime.  Both $E_{\nu,s}$ and $S_N$ may grow like $\log|t|$, so boundedness depends on the directional coefficients, not merely on whether $E_{\nu,s}\to\infty$.  The exact test is the entropy-profile criterion in \cref{prop:Theta-criterion}.  A kernel direction is still fatal, while sufficiently large logarithmic coefficients can make $F_\beta$ bounded below at a given temperature.

\subsubsection{Bounded redescending losses}

Tukey's biweight is a standard example:
\begin{equation*}
\rho_c(u)
=
\begin{cases}
\dfrac{c^2}{6}\left[1-\left(1-(u/c)^2\right)^3\right],&|u|\leq c,\\[2mm]
\dfrac{c^2}{6},&|u|>c,
\end{cases}
\qquad
E_{\mathrm T}(W)=\sum_{j=1}^M\rho_c(r_j(W)).
\end{equation*}
Redescending robust procedures originate in work such as \cite{BeatonTukey1974}.

\emph{Properties of the energy.}
For a finite dataset, $E_{\mathrm T}$ is bounded above on the entire parameter space.  It may possess useful finite local or global minimizers, but it supplies no confinement at infinity.

\emph{Conclusion for the free energy.}
Along every nontrivial affine ray through a full-rank matrix, the energy stays bounded while the fiber entropy diverges.  Hence $F_\beta$ is unbounded below for every finite $\beta$.  A bounded redescending loss must be accompanied by an explicit confining penalty or a proper prior if one wants a well-posed free-energy minimization problem.

\subsection{Log-determinant energies on the positive-definite cone}\label{subsec:logdet-energy}

\subsubsection{Gaussian covariance and precision estimation}

For a positive-definite precision matrix $\Theta$, the Gaussian negative log-likelihood and its graphical-lasso modification are
\begin{equation*}
E_{\mathrm{glasso}}(\Theta)
=
\Tr(\widehat\Sigma\Theta)
-
\log\det\Theta
+
\lambda\|\Theta\|_{1,\mathrm{off}},
\qquad \Theta\succ0.
\end{equation*}
This energy is used to estimate sparse conditional-dependence graphs \cite{FriedmanHastieTibshirani2008}.

\emph{Properties of the energy.}
If $\widehat\Sigma\succ0$, then $\Tr(\widehat\Sigma\Theta)$ grows at least linearly as $\|\Theta\|\to\infty$, while $-\log\det\Theta\to+\infty$ as $\Theta$ approaches the singular boundary.  Thus the energy has compact sublevel sets in the positive-definite cone and attains a positive-definite minimum.  If $\widehat\Sigma$ is singular, coercivity can fail in null directions and must be restored by the penalty or an additional hypothesis.

\emph{Conclusion for the free energy.}
When $\widehat\Sigma\succ0$, the linear confinement dominates the logarithmic fiber entropy at infinity, and the log-determinant term supplies precisely the rank-boundary barrier required by \cref{cor:attainment-condition}.  The free energy therefore has compact sublevel sets and an attained full-rank minimum.  This is a model example in which both possible escape mechanisms are excluded.

\subsection{A practical taxonomy}
\label{subsec:practical-taxonomy}

The examples fall into four broad classes.

\subsubsection{Superlogarithmically confining energies}

Full least squares, injective matrix sensing, ridge regularization, identifiable Huber regression, quadratically regularized margin losses, and positive-definite covariance likelihoods grow linearly or quadratically in every unbounded direction.  They satisfy
\[
\frac{E(W)}{\log(1+\normF{W})}\longrightarrow+\infty,
\]
so $F_\beta$ is norm-coercive for every finite temperature.  Actual attainment in $\GLd$ may still require a rank-boundary argument unless the energy already contains a barrier.

\subsubsection{Underdetermined or gauge-invariant energies}

Matrix completion, noninjective matrix sensing, rank-deficient regression, and unfixed softmax parametrizations have nontrivial affine directions on which the loss is constant.  Their ordinary energies may have excellent finite minimizers, but \cref{cor:affine-energy-obstruction} forces the corresponding free energy to be unbounded below.  Restricting the model, quotienting the symmetry, or adding a confining prior is essential.

\subsubsection{Losses whose optima occur at infinity or on unbounded minimizing sets}

Separable logistic and exponential losses approach their infimum along separating rays, while separable hinge loss vanishes on an unbounded cone.  The former may have no finite minimizer and the latter may have many, but in either case the energy remains bounded along an escape direction and the entropy sends $F_\beta$ to $-\infty$.

\subsubsection{Borderline logarithmic and rank-sensitive energies}

Student-$t$ and Cauchy negative log-likelihoods can grow on the same logarithmic scale as $S_N$.  Nuclear-norm penalties, factor-induced Schatten penalties, and other low-rank regularizers are confining at infinity but can compete with the entropy near the rank boundary.  Here boundedness and attainment depend on temperature, depth, coefficients, singular-value geometry, and the observation map.  The entropy profile $\Theta_E$ from \cref{prop:Theta-criterion} is the appropriate global test.

The practical rule is
\begin{equation*}
\begin{array}{rcl}
\text{flat or bounded-energy escape directions}
&\Longrightarrow&
F_\beta\text{ is unbounded below},\\[1mm]
\text{superlogarithmic growth of }E
&\Longrightarrow&
F_\beta\text{ is norm-coercive},\\[1mm]
\text{logarithmic growth of }E
&\Longrightarrow&
\text{compare directional coefficients},\\[1mm]
\text{a minimum in }\GLd
&\Longrightarrow&
\text{also control the rank boundary}
\end{array}
\end{equation*}
This taxonomy separates two failures of deterministic entropic selection.  The energy may fail to confine the system at infinity, or it may be norm-coercive while still allowing minimizing sequences to approach corank one. Entropic selection is well posed only after these global defects have been excluded.

\section{The free energy is unbounded below}\label{sec:unbounded}

The diagonal observation operator is noninjective, so unboundedness already follows abstractly from \cref{thm:linear-sensing-dichotomy}.  We nevertheless give the explicit zero-loss shear because it identifies the precise logarithmic escape rate of the entropy and will be useful in interpreting the dynamics.

The noncompact zero-loss set already prevents free-energy minimization.

\begin{proposition}[Zero-loss shear]\label{prop:shear}
For $a\in\mathbb R$, let
\begin{equation*}
W_a=I_d+a e_1e_2^T.
\end{equation*}
Then $E_d(W_a)=0$, and the singular values of $W_a$ are
\[
s_a,
\underbrace{1,\ldots,1}_{d-2\text{ times}},
s_a^{-1},
\]
where
\begin{equation}\label{eq:sa}
s_a=\frac{\sqrt{a^2+4}+|a|}{2}.
\end{equation}
In particular, $s_a\sim|a|$ as $|a|\to\infty$.
\end{proposition}

\begin{proof}
Only the upper-left $2\times2$ block differs from the identity.  That block is
\[
\begin{pmatrix}1&a\\0&1\end{pmatrix}.
\]
The product of its singular values is the absolute determinant, equal to $1$, while the sum of their squares is its squared Frobenius norm, equal to $a^2+2$.  Solving these two equations gives $s_a$ and $s_a^{-1}$ with \eqref{eq:sa}.  All remaining singular values are $1$.
\end{proof}

\begin{theorem}[Failure of free-energy minimization]\label{thm:unbounded}
For every $d\geq2$, every integer $N\geq2$, and every finite $\beta>0$,
\[
\inf_{W\in\GLd}F_\beta(W)=-\infty.
\]
More precisely,
\begin{equation}\label{eq:shear-entropy-asymptotic}
S_N(W_a)=(d-1)\cN\log s_a+O_{N,d}(1)
\qquad
\text{as }|a|\to\infty.
\end{equation}
\end{theorem}

\begin{proof}
There are $d-2$ pairs $(s_a,1)$ and one pair $(s_a,s_a^{-1})$.  By \eqref{eq:q-ratio}, each contributes $\cN\log s_a+O(1)$.  The $d-2$ pairs $(1,s_a^{-1})$ and all pairs of unit singular values remain bounded.  This proves \eqref{eq:shear-entropy-asymptotic}.  Since $E_d(W_a)=0$,
\[
F_\beta(W_a)=-\beta^{-1}S_N(W_a)\longrightarrow-\infty.
\]
\end{proof}

\begin{remark}
The argument is not tied to the target diagonal $(1,\ldots,1)$.  If every prescribed diagonal entry is nonzero, adding a large off-diagonal shear preserves the observed entries and invertibility while producing the same logarithmic entropy growth.
\end{remark}

\section{Classification of all full-rank critical points}\label{sec:critical}

Because $\Aop_{N,W}$ is invertible on $\GLd$, equilibria of \eqref{eq:free-flow} are exactly Euclidean critical points of $F_\beta$.

\subsection{The critical-point equation}

At a critical point,
\begin{equation}\label{eq:critical-equation}
E_d'(W)=\frac1\beta S_N'(W).
\end{equation}
The left-hand side is diagonal.  The right-hand side is spectral.  Positivity and strict concavity force these structures to align.

\begin{lemma}[Shared polar factor]\label{lem:shared-polar}
Let $W=U\Sigma V^T\in\GLd$, and let $g_i=\partial_i\Sn(\sigma)$.  Then
\[
S_N'(W)=U\diag(g_1,\ldots,g_d)V^T.
\]
Moreover, $S_N'(W)$ is invertible and has the same orthogonal polar factor $UV^T$ as $W$.
\end{lemma}

\begin{proof}
The gradient formula is \eqref{eq:spectral-gradient}.  By \cref{cor:g-positive}, every $g_i$ is positive.  Thus the displayed decomposition is an SVD of $S_N'(W)$ with the same left and right singular frames as $W$, and its orthogonal polar factor is $UV^T$.
\end{proof}

\begin{theorem}[Critical matrices are diagonal]\label{thm:critical-diagonal}
Assume $N>2$.  Every critical point of $F_\beta$ in $\GLd$ is diagonal.
\end{theorem}

\begin{proof}
Let
\[
P=E_d'(W)=\diag(W_{11}-1,\ldots,W_{dd}-1).
\]
At a critical point, \eqref{eq:critical-equation} and \cref{lem:shared-polar} show that $P$ is invertible.  Write the right polar decomposition of $W$ as
\[
W=QH,
\qquad
Q=UV^T,
\qquad
H=V\Sigma V^T>0.
\]
The entropy gradient has the corresponding polar decomposition
\[
S_N'(W)=QG,
\qquad
G=V\diag(g_1,\ldots,g_d)V^T>0.
\]
Equation \eqref{eq:critical-equation} gives
\[
P=\frac1\beta QG.
\]
The polar decomposition of an invertible matrix is unique.  Since $P$ is diagonal, its orthogonal polar factor is the diagonal sign matrix
\[
D_\epsilon=\diag(\sgn P_{11},\ldots,\sgn P_{dd}).
\]
Therefore $Q=D_\epsilon$, and
\[
G=\beta D_\epsilon P
\]
is diagonal and positive.

The positive matrices $G$ and $H$ commute because they are simultaneously diagonalized by $V$.  Partition the standard coordinate space into eigenspaces of the diagonal matrix $G$.  The commutation relation makes $H$ block diagonal with respect to this partition.  Within one such block, all eigenvalues $g_i$ of $G$ agree.  By \cref{lem:gradient-order}, the corresponding singular values $\sigma_i$ agree.  Hence every eigenvalue of the restriction of $H$ to the block is the same positive number, so that restriction is a scalar multiple of the identity.  It follows that $H$ is diagonal in the standard basis.  Since $Q=D_\epsilon$ is also diagonal, $W=QH$ is diagonal.
\end{proof}

\subsection{One critical point in each sign chamber}

Write a diagonal invertible matrix as
\begin{equation*}
W=\diag(\epsilon_1x_1,\ldots,\epsilon_dx_d),
\qquad
x_i>0,
\quad
\epsilon_i\in\{-1,1\}.
\end{equation*}
On this chamber,
\[
S_N'(W)=\diag(\epsilon_1g_1(x),\ldots,\epsilon_dg_d(x)),
\]
so \eqref{eq:critical-equation} is equivalent to
\begin{equation}\label{eq:critical-x}
x_i-\epsilon_i=\frac1\beta g_i(x),
\qquad i=1,\ldots,d.
\end{equation}
These are the Euler equations for
\begin{equation}\label{eq:Phi-eps}
\Phi_\epsilon(x)
=
\frac12\sum_{i=1}^d(x_i-\epsilon_i)^2
-
\frac1\beta\Sn(x),
\qquad x\in(0,\infty)^d.
\end{equation}

\begin{lemma}[Boundary behavior of $\Phi_\epsilon$]\label{lem:Phi-boundary}
Assume $N>2$.
\begin{enumerate}[label=\textup{(\roman*)}]
\item $\Phi_\epsilon(x)\to+\infty$ as $\lVert x\rVert\to\infty$.
\item If at least two coordinates tend to zero while $x$ remains bounded, then $\Phi_\epsilon(x)\to+\infty$.
\item At a boundary point with exactly one zero coordinate and all others positive, the inward one-sided derivative of $\Phi_\epsilon$ in that coordinate is $-\infty$.
\end{enumerate}
\end{lemma}

\begin{proof}
The logarithmic upper bound in \cref{lem:entropy-upper}, applied to a diagonal matrix, shows that the quadratic term in \eqref{eq:Phi-eps} dominates at infinity.

If $x_i,x_j\to0$, use \eqref{eq:q-ratio} with $a=\max\{x_i,x_j\}$ to obtain
\[
q_N(x_i,x_j)=\cN\log a+O(1)\longrightarrow-\infty.
\]
All other pair terms are bounded above on a bounded set.  Thus $\Sn(x)\to-\infty$ and $\Phi_\epsilon(x)\to+\infty$.

Finally, suppose $x_i\downarrow0$ while $x_j$ remains in a compact subset of $(0,\infty)$ for $j\neq i$.  By \cref{lem:boundary-derivative},
\[
g_i(x)\longrightarrow+\infty.
\]
Hence
\[
\partial_i\Phi_\epsilon(x)
=x_i-\epsilon_i-\frac1\beta g_i(x)
\longrightarrow-\infty.
\]
\end{proof}

\begin{theorem}[Exactly $2^d$ critical points]\label{thm:2d-critical}
Assume $N>2$.  For every $\epsilon\in\{-1,1\}^d$, the function $\Phi_\epsilon$ has a unique critical point $x^\epsilon\in(0,\infty)^d$.  Consequently,
\[
\#\Crit_{\GLd}(F_\beta)=2^d,
\]
and the critical points are exactly
\[
W_\epsilon=\diag(\epsilon_1x_1^\epsilon,\ldots,\epsilon_dx_d^\epsilon).
\]
\end{theorem}

\begin{proof}
By \cref{thm:strict-concavity},
\begin{equation*}
\nabla^2\Phi_\epsilon(x)
=I-\frac1\beta\nabla^2\Sn(x)
\succ0.
\end{equation*}
Thus $\Phi_\epsilon$ is strongly convex and has at most one critical point.

To prove existence, extend $\Phi_\epsilon$ to the closed positive orthant by taking the finite limit at boundary points with exactly one zero coordinate and setting the value equal to $+\infty$ when at least two coordinates vanish.  By \cref{lem:Phi-boundary}, this lower-semicontinuous extension is coercive, so it attains a minimum.  A boundary point with at least two zeros has infinite value.  A boundary point with exactly one zero cannot minimize because moving slightly inward lowers $\Phi_\epsilon$.  Hence the minimizer lies in $(0,\infty)^d$ and is the unique critical point.

The diagonality theorem now shows that the resulting $2^d$ points exhaust the full critical set.
\end{proof}

\begin{remark}[Permutation types]\label{rem:permutation-types}
If a sign vector has $k$ positive entries, uniqueness and permutation symmetry imply that all positive coordinates of $x^\epsilon$ have one common magnitude and all negative coordinates have another.  Hence there are only $d+1$ critical-point types up to coordinate permutation, although there are $2^d$ distinct points.
\end{remark}

\begin{proposition}[The two scalar equilibria]\label{prop:scalar-equilibria}
Let
\[
C_{d,N,\beta}=\frac{(d-1)\cN}{2\beta}.
\]
The all-positive and all-negative critical points are
\[
W_+=w_+I_d,
\qquad
W_-=-w_-I_d,
\]
where
\begin{equation*}
w_+(w_+-1)=C_{d,N,\beta},
\qquad
w_-(w_-+1)=C_{d,N,\beta}.
\end{equation*}
Thus
\[
w_+=\frac{1+\sqrt{1+4C_{d,N,\beta}}}{2},
\qquad
w_-=\frac{-1+\sqrt{1+4C_{d,N,\beta}}}{2}.
\]
\end{proposition}

\begin{proof}
At $x=w\bm 1$, symmetry and the homogeneity
\[
\Sn(t\sigma)=\Sn(\sigma)+\binom d2\cN\log t
\]
give
\[
g_i(w\bm1)=\frac{(d-1)\cN}{2w}.
\]
Substitute this into \eqref{eq:critical-x} for the all-positive and all-negative sign vectors.
\end{proof}

\begin{proposition}[Zero-temperature location of the saddles]\label{prop:zero-temp-location}
For a fixed sign vector $\epsilon$, as $\beta\to\infty$,
\[
x_i^\epsilon(\beta)\longrightarrow
\begin{cases}
1,&\epsilon_i=1,\\
0,&\epsilon_i=-1.
\end{cases}
\]
Equivalently,
\[
W_\epsilon(\beta)\longrightarrow
\diag\left(\frac{1+\epsilon_1}{2},\ldots,\frac{1+\epsilon_d}{2}\right).
\]
\end{proposition}

\begin{proof}
Let $m$ be the number of negative entries of $\epsilon$, and let $x^0$ denote the boundary vector with coordinates $1$ at positive signs and $0$ at negative signs.  The quadratic part of \eqref{eq:Phi-eps} has the unique minimum $m/2$ on the closed positive orthant, attained at $x^0$.

Choose a competitor $x^{(\beta)}$ with positive coordinates equal to $1$ and negative coordinates equal to $\beta^{-1}$.  Its quadratic energy tends to $m/2$, while \cref{lem:q-basic} gives $|\Sn(x^{(\beta)})|=O(\log\beta)$.  Hence
\[
\limsup_{\beta\to\infty}\Phi_\epsilon(x^\epsilon(\beta))\leq\frac m2.
\]
The logarithmic upper bound for $\Sn$ makes the family $x^\epsilon(\beta)$ bounded.  On a fixed bounded set, $\Sn$ is bounded above, so
\[
\Phi_\epsilon(x)\geq
\frac12\sum_i(x_i-\epsilon_i)^2-O(\beta^{-1}).
\]
It follows that the quadratic term at $x^\epsilon(\beta)$ tends to its minimum $m/2$.  Strict convexity of that quadratic function on the closed orthant forces $x^\epsilon(\beta)\to x^0$.
\end{proof}

\section{The full Hessian and the Morse index}\label{sec:hessian}

The restriction of $F_\beta$ to a diagonal sign chamber is strongly convex, but each critical point becomes a saddle when off-diagonal perturbations are allowed.  We first record the relevant second-order formula.

\subsection{A spectral Hessian formula}

Let $f:(0,\infty)^d\to\mathbb R$ be a symmetric $C^2$ function and define $\Phi(W)=f(\sigma(W))$.  Fix
\[
W_*=D_\epsilon X,
\qquad
D_\epsilon=\diag(\epsilon_1,\ldots,\epsilon_d),
\qquad
X=\diag(x_1,\ldots,x_d)>0.
\]
Put $g_i=\partial_if(x)$.

\begin{lemma}[Second variation at a signed diagonal matrix]\label{lem:spectral-Hessian}
The matrix space splits into the Frobenius-orthogonal sum
\begin{equation}\label{eq:matrix-splitting}
\Md=\mathcal D\oplus\bigoplus_{i<j}\mathcal V_{ij},
\qquad
\mathcal D=\{\text{diagonal matrices}\},
\quad
\mathcal V_{ij}=\operatorname{span}\{E_{ij},E_{ji}\},
\end{equation}
and the Hessian of $\Phi$ respects this splitting.

For $H=\diag(h_1,\ldots,h_d)$,
\begin{equation}\label{eq:hessian-diagonal}
D^2\Phi(W_*)[H,H]
=(D_\epsilon h)^T\nabla^2f(x)(D_\epsilon h).
\end{equation}
For $H=aE_{ij}+bE_{ji}$,
\begin{equation}\label{eq:hessian-offdiag}
D^2\Phi(W_*)[H,H]
=
\frac{A_{ij}}2(\epsilon_i a+\epsilon_j b)^2
+
\frac{B_{ij}}2(\epsilon_i a-\epsilon_j b)^2,
\end{equation}
where
\begin{equation*}
A_{ij}=\frac{g_i-g_j}{x_i-x_j},
\qquad
B_{ij}=\frac{g_i+g_j}{x_i+x_j}.
\end{equation*}
At $x_i=x_j$, $A_{ij}$ is interpreted by its continuous limit.
\end{lemma}

\begin{proof}
The diagonal formula follows because, for sufficiently small $t$, the singular values of
\[
W_*+tH=\diag(\epsilon_ix_i+th_i)
\]
are $x_i+t\epsilon_i h_i$.

For the off-diagonal formula, pass to the singular frame of $W_*$, where the perturbation has entries
\[
k_{ij}=\epsilon_i a,
\qquad
k_{ji}=\epsilon_j b.
\]
It is enough to work in the $(i,j)$ block.  Assume first that $x_i\neq x_j$.  For
\[
X_{ij}(t)=
\begin{pmatrix}
 x_i&t k_{ij}\\
 t k_{ji}&x_j
\end{pmatrix},
\]
the eigenvalues of $X_{ij}(t)^TX_{ij}(t)$ give the second derivatives of the two singular values.  A direct $2\times2$ perturbation calculation yields
\begin{align*}
D^2\Phi(W_*)[H,H]
&=
\frac12\frac{g_i-g_j}{x_i-x_j}(k_{ij}+k_{ji})^2\\
&\quad+
\frac12\frac{g_i+g_j}{x_i+x_j}(k_{ij}-k_{ji})^2.
\end{align*}
Substitution of $k_{ij},k_{ji}$ gives \eqref{eq:hessian-offdiag}.  Orthogonal invariance, or polarization of the same calculation, shows that distinct spaces in \eqref{eq:matrix-splitting} do not couple.  The repeated-value formula follows by continuity of the Hessian of a smooth symmetric singular-value function; see also \cite{LewisSendov2001,DrusvyatskiyKempton2015}.
\end{proof}

Apply \cref{lem:spectral-Hessian} to $f=\Sn$ at a critical point $W_\epsilon$.

\begin{lemma}[Signs of the off-diagonal entropy modes]\label{lem:AB-signs}
Assume $N>2$.  At every positive $x$,
\[
A_{ij}<0,
\qquad
B_{ij}>0.
\]
\end{lemma}

\begin{proof}
If $x_i\neq x_j$, the first inequality is \cref{lem:gradient-order}.  If $x_i=x_j$, it follows from negative definiteness of $\nabla^2\Sn$ by taking the second derivative in the $e_i-e_j$ direction.  The second inequality follows from $g_i,g_j>0$ and $x_i,x_j>0$.
\end{proof}

\begin{theorem}[Universal Morse index]\label{thm:Morse-index}
At every critical point $W_\epsilon$, the Hessian of $F_\beta$ has
\[
\frac{d(d+1)}2
\]
positive eigenvalues and
\[
\binom d2
\]
negative eigenvalues.  It has no zero eigenvalues.
\end{theorem}

\begin{proof}
On the diagonal subspace $\mathcal D$, the energy Hessian is the identity.  By \eqref{eq:hessian-diagonal} and \cref{thm:strict-concavity}, the entropy Hessian is negative definite after a diagonal sign congruence.  Hence
\[
D^2F_\beta|_{\mathcal D}
=D^2E_d|_{\mathcal D}-\frac1\beta D^2S_N|_{\mathcal D}
\]
is positive definite.  This gives $d$ positive directions.

On $\mathcal V_{ij}$, the energy Hessian vanishes.  By \eqref{eq:hessian-offdiag} and \cref{lem:AB-signs}, the entropy Hessian has one negative and one positive eigenvalue.  Therefore $-\beta^{-1}D^2S_N$ has one positive and one negative eigenvalue on each $\mathcal V_{ij}$.  There are $\binom d2$ such pair spaces.  Adding the counts gives
\[
d+\binom d2=\frac{d(d+1)}2
\]
positive directions and $\binom d2$ negative directions, with no kernel.
\end{proof}

The unstable off-diagonal line associated with the pair $(i,j)$ is particularly explicit.  It is the line on which
\[
\epsilon_i a+\epsilon_j b=0,
\]
namely
\begin{equation*}
\operatorname{span}\{E_{ij}-\epsilon_i\epsilon_jE_{ji}\}.
\end{equation*}
The complementary off-diagonal line
\[
\operatorname{span}\{E_{ij}+\epsilon_i\epsilon_jE_{ji}\}
\]
is stable at the Hessian level.

\subsection{Linearization of the Riemannian flow}

Let
\[
\mathcal X(W)=-\Aop_{N,W}(F_\beta'(W))
\]
denote the vector field.  At a critical point, the derivative of $\Aop_{N,W}$ is multiplied by $F_\beta'(W_\epsilon)=0$.  Hence
\begin{equation}\label{eq:linearization}
D\mathcal X(W_\epsilon)
=-\Aop_{N,W_\epsilon}\circ D^2F_\beta(W_\epsilon).
\end{equation}

\begin{corollary}[Hyperbolicity and stable dimensions]\label{cor:hyperbolic}
Every $W_\epsilon$ is a hyperbolic equilibrium of the free-energy flow.  Its stable and unstable dimensions are
\[
\dim E^s=\frac{d(d+1)}2,
\qquad
\dim E^u=\binom d2.
\]
\end{corollary}

\begin{proof}
The operator $\Aop_{N,W_\epsilon}$ is Frobenius self-adjoint and positive definite.  Therefore
\[
\Aop_{N,W_\epsilon}D^2F_\beta(W_\epsilon)
\]
is similar to the symmetric matrix
\[
\Aop_{N,W_\epsilon}^{1/2}
D^2F_\beta(W_\epsilon)
\Aop_{N,W_\epsilon}^{1/2}.
\]
By Sylvester's law of inertia, the latter has the same inertia as $D^2F_\beta(W_\epsilon)$.  The minus sign in \eqref{eq:linearization} makes positive Hessian directions dynamically stable and negative Hessian directions unstable.  The result follows from \cref{thm:Morse-index}.
\end{proof}

\section{The rank boundary cannot be a finite limit}\label{sec:rank-boundary}

\subsection{The Lyapunov identity}

Along the free-energy flow,
\begin{align}
\frac{d}{dt}F_\beta(W(t))
&=
\ip{F_\beta'(W)}{\dot W}\notag\\
&=
-\ip{F_\beta'(W)}{\Aop_{N,W}F_\beta'(W)}\notag\\
&=
-\normg{\grad_{g^N}F_\beta(W)}^2
\leq0.
\label{eq:Lyapunov}
\end{align}
Thus $F_\beta$ is a strict Lyapunov function on $\GLd$.

\subsection{An infinite barrier at corank at least two}

\begin{proposition}[Corank at least two]\label{prop:corank-two}
Let $W(t)$ be a full-rank solution.  It cannot converge to a finite matrix $W_*$ with
\[
\rank W_*\leq d-2.
\]
\end{proposition}

\begin{proof}
If $W(t)\to W_*$ with rank at most $d-2$, then
\[
\sigma_{d-1}(t)\to0,
\qquad
\sigma_d(t)\to0,
\]
while all singular values remain bounded.  By \eqref{eq:q-ratio},
\[
q_N(\sigma_{d-1},\sigma_d)
\leq
\cN\log\sigma_{d-1}+C
\longrightarrow-\infty.
\]
Every other pair term is bounded above along the bounded trajectory.  Hence
\[
S_N(W(t))\longrightarrow-\infty.
\]
The energy has a finite limit, so $F_\beta(W(t))\to+\infty$, contradicting the monotonicity \eqref{eq:Lyapunov}.
\end{proof}

\begin{remark}
This argument is stronger than a local vector-field calculation.  The full-rank free energy has an infinite barrier at every finite matrix of corank at least two.
\end{remark}

\subsection{Repulsion from the corank-one boundary}

Suppose the smallest singular value is simple.  Let $u_i,v_i$ be corresponding singular vectors and set
\begin{equation*}
e_i(W)=u_i^TE_d'(W)v_i.
\end{equation*}
By \eqref{eq:A-spectrum} and the standard singular-value derivative formula,
\begin{equation}\label{eq:sigma-flow}
\dot\sigma_i
=-N\sigma_i^{2-2/N}e_i
+
\frac N\beta\sigma_i^{2-2/N}g_i.
\end{equation}

\begin{proposition}[Corank-one repulsion]\label{prop:corank-one}
Let $N>2$.  A full-rank solution cannot converge to a finite matrix of rank $d-1$.
\end{proposition}

\begin{proof}
Suppose that $W(t)\to W_*$ and $\rank W_*=d-1$.  Then
\[
\sigma_d(t)\to0,
\]
while $\sigma_1(t),\ldots,\sigma_{d-1}(t)$ converge to positive numbers.  For large $t$, $\sigma_d$ is isolated and differentiable.  Put $b=\sigma_d$ and $a_j=\sigma_j$ for $j<d$.  By \cref{lem:boundary-derivative},
\begin{equation*}
g_d
=
\frac1N b^{2/N-1}
\sum_{j<d}a_j^{-2/N}
\bigl(1+o(1)\bigr).
\end{equation*}
The quantity $e_d$ is bounded because $W(t)$ is bounded.  Substitution into \eqref{eq:sigma-flow} yields
\begin{equation}\label{eq:sigmad-repulsion}
\dot b
=
\frac1\beta
\left(\sum_{j<d}a_j^{-2/N}\right)b
-
N e_d b^{2-2/N}
+o(b).
\end{equation}
Since $N>2$,
\[
2-\frac2N>1,
\]
so $b^{2-2/N}=o(b)$.  The positive linear term in \eqref{eq:sigmad-repulsion} dominates.  Hence there are $c>0$ and $T$ such that
\[
\dot\sigma_d(t)\geq c\sigma_d(t)>0
\qquad
\text{for }t\geq T.
\]
This is incompatible with $\sigma_d(t)\to0$.
\end{proof}

\begin{corollary}[No finite singular limit]\label{cor:no-singular-limit}
For $N>2$, no solution starting in $\GLd$ converges to a finite rank-deficient matrix.
\end{corollary}

\begin{remark}[Why $N=2$ is different]\label{rem:N2-boundary}
At depth $N=2$, the energy term in \eqref{eq:sigmad-repulsion} is also linear in $\sigma_d$.  The entropy term no longer dominates automatically, and the sign of the total coefficient depends on the state.  This is one of the two places where the proof of \cref{thm:main} genuinely uses $N>2$.
\end{remark}

\section{Forward completeness}\label{sec:complete}

The vector field is smooth on $\GLd$, but the metric is incomplete at the rank boundary and the free energy is unbounded below at infinity.  We therefore prove forward completeness directly.

\subsection{A metric lower bound for radial escape}

\begin{lemma}[Radial length estimate]\label{lem:radial-length}
Let $W:[a,b]\to\GLd$ be an absolutely continuous curve.  Then
\begin{equation}\label{eq:radial-length}
\operatorname{Length}_{g^N}(W|_{[a,b]})
\geq
\sqrt N\,
\left|
\sigma_1(W(b))^{1/N}-\sigma_1(W(a))^{1/N}
\right|.
\end{equation}
\end{lemma}

\begin{proof}
The function $t\mapsto\sigma_1(W(t))$ is absolutely continuous and differentiable almost everywhere.  At a time when the top singular value is simple, the usual first-variation formula and \eqref{eq:A-spectrum} give
\[
|\dot\sigma_1|^2
\leq
N\sigma_1^{2-2/N}\normg{\dot W}^2.
\]
If the top singular value has multiplicity $m>1$, its directional derivatives are eigenvalues of the symmetric part of the corresponding $m\times m$ singular-frame block of $\dot W$.  Their absolute values are bounded by the Frobenius norm of that block.  On the whole top block, including its off-diagonal entries, the continuous value of the coefficient in \eqref{eq:A-spectrum} is $N\sigma_1^{2-2/N}$.  The same inequality therefore holds at every differentiability time; see also the standard singular-value differential calculus in \cite{Lewis1996,LewisSendov2001}.

It follows that
\[
\left|\frac{d}{dt}\sigma_1^{1/N}\right|
=\frac1N\sigma_1^{1/N-1}|\dot\sigma_1|
\leq
\frac1{\sqrt N}\normg{\dot W}
\]
for almost every $t$.  Integration proves \eqref{eq:radial-length}.
\end{proof}

\begin{lemma}[Length-energy estimate]\label{lem:length-energy}
Along a solution of \eqref{eq:free-flow},
\begin{equation*}
\operatorname{Length}_{g^N}(W|_{[0,T]})^2
\leq
T\bigl(F_\beta(W(0))-F_\beta(W(T))\bigr).
\end{equation*}
\end{lemma}

\begin{proof}
The speed is
\[
\normg{\dot W}=\normg{\grad_{g^N}F_\beta}.
\]
Cauchy-Schwarz and \eqref{eq:Lyapunov} give
\begin{align*}
\operatorname{Length}_{g^N}(W|_{[0,T]})^2
&\leq
T\int_0^T\normg{\grad_{g^N}F_\beta}^2\,dt\\
&=
T\bigl(F_\beta(W(0))-F_\beta(W(T))\bigr).
\end{align*}
\end{proof}

\begin{proposition}[No finite-time escape to infinity]
  \label{prop:no-finite-infinity}
  A solution of \eqref{eq:free-flow} satisfies
  $$\sup_{t<T}\|W(t)\|<\infty$$
  for all $T>0$.
\end{proposition}

\begin{proof}
Suppose a maximal solution has finite terminal time $T_{\max}$ and unbounded norm.  Since $\normF{W}$ is equivalent to $\sigma_1(W)$ up to a dimension-dependent factor, for arbitrarily large $R$ there is a first time $t_R<T_{\max}$ with $\sigma_1(W(t_R))=R$.

By \cref{lem:entropy-upper} and $E_d\geq0$,
\begin{align*}
F_\beta(W(0))-F_\beta(W(t_R))
&\leq
F_\beta(W(0))+\frac1\beta S_N(W(t_R))\\
&\leq
C_0+C_1\log(1+R)
\end{align*}
for constants independent of $R$.  Combining \cref{lem:radial-length,lem:length-energy} gives
\[
\sqrt N\left(R^{1/N}-\sigma_1(W(0))^{1/N}\right)
\leq
\sqrt{T_{\max}\bigl(C_0+C_1\log(1+R)\bigr)}.
\]
The left side grows like $R^{1/N}$ and the right side like $\sqrt{\log R}$, a contradiction.
\end{proof}

\subsection{No finite-time collision with the rank boundary}

\begin{lemma}[A lower entropy bound on trajectories]\label{lem:entropy-lower-trajectory}
Along every solution,
\begin{equation*}
S_N(W(t))\geq-\beta F_\beta(W(0)).
\end{equation*}
\end{lemma}

\begin{proof}
Since $F_\beta(W(t))\leq F_\beta(W(0))$ and $E_d\geq0$,
\[
S_N(W(t))
=\beta\bigl(E_d(W(t))-F_\beta(W(t))\bigr)
\geq-\beta F_\beta(W(0)).
\]
\end{proof}

\begin{lemma}[Two small singular values are impossible on a bounded sublevel]\label{lem:two-small}
Fix $M,L>0$.  There exists $\delta>0$ such that
\[
\sigma_1(W)\leq M,
\qquad
S_N(W)\geq-L
\]
imply
\[
\sigma_{d-1}(W)\geq\delta.
\]
\end{lemma}

\begin{proof}
Otherwise there is a bounded sequence with $\sigma_{d-1}\to0$.  Then also $\sigma_d\to0$, and the pair term $q_N(\sigma_{d-1},\sigma_d)$ tends to $-\infty$ by \eqref{eq:q-ratio}.  All remaining pair terms are bounded above because $\sigma_1\leq M$.  Thus $S_N\to-\infty$, contradicting $S_N\geq-L$.
\end{proof}

\begin{proposition}[No finite-time rank loss]\label{prop:no-finite-rank-loss}
A solution of \eqref{eq:free-flow} cannot reach the rank boundary in finite time.
\end{proposition}

\begin{proof}
Suppose a maximal solution remains bounded on $[0,T_{\max})$ with $T_{\max}<\infty$.  By \cref{lem:entropy-lower-trajectory,lem:two-small}, there is $m>0$ such that
\[
\sigma_{d-1}(W(t))\geq m
\]
throughout the interval.  If $\sigma_d$ were to approach zero, it would eventually be simple.  Since $g_d>0$, equation \eqref{eq:sigma-flow} gives
\[
\dot\sigma_d
\geq
-N|e_d|\sigma_d^{p},
\qquad
p=2-\frac2N>1.
\]
Boundedness of $W$ makes $|e_d|$ uniformly bounded by a constant $C$.  Hence, after some time $t_0$,
\[
\dot\sigma_d\geq-NC\sigma_d^p.
\]
Since $p>1$, this differential inequality yields
\[
\sigma_d(t)
\geq
\left(
\sigma_d(t_0)^{1-p}+NC(p-1)(t-t_0)
\right)^{-1/(p-1)}>0
\]
for every finite $t\geq t_0$.  Thus a positive singular value cannot reach zero in finite time.  The trajectory remains in a compact subset of $\GLd$, so the smooth ODE extends past $T_{\max}$, a contradiction.
\end{proof}

\begin{theorem}[Forward completeness]\label{thm:forward-complete}
Every solution of \eqref{eq:free-flow} with initial condition in $\GLd$ exists for all $t\geq0$ and remains full rank.  In particular, the sign of $\det W(t)$ is constant.
\end{theorem}

\begin{proof}
A finite maximal time would force either unbounded norm or approach to the boundary of $\GLd$.  These alternatives are excluded by \cref{prop:no-finite-infinity,prop:no-finite-rank-loss}.  Continuity of the nonzero determinant then preserves its sign.
\end{proof}

\section{Bounded orbits and generic unboundedness}\label{sec:generic}

\subsection{A bounded orbit is precompact in \texorpdfstring{$\GLd$}{GLd}}

\begin{proposition}[Uniform separation from the rank boundary]\label{prop:bounded-precompact}
If a forward orbit is bounded in $\Md$, then there is $\delta>0$ such that
\[
\sigma_d(W(t))\geq\delta
\qquad
\text{for all }t\geq0.
\]
Consequently, the orbit has compact closure in $\GLd$.
\end{proposition}

\begin{proof}
Let $\sigma_1(W(t))\leq M$.  By \cref{lem:entropy-lower-trajectory,lem:two-small},
\[
\sigma_{d-1}(W(t))\geq m>0.
\]
It remains to control $b=\sigma_d$.

When $0<b<m/2$, the smallest singular value is simple.  For $a_j=\sigma_j\in[m,M]$, use the exact formula \eqref{eq:boundary-drift-exact} in \eqref{eq:sigma-flow}.  Uniformly as $b\downarrow0$,
\begin{equation*}
\frac N\beta b^{2-2/N}g_d
=
\frac1\beta\left(\sum_{j<d}a_j^{-2/N}\right)b+o(b).
\end{equation*}
The energy contribution is bounded below by
\[
-NC b^{2-2/N}=o(b).
\]
Thus there are $\delta\in(0,m/2)$ and $c>0$ such that
\begin{equation}\label{eq:outward-barrier}
0<b\leq\delta
\quad\Longrightarrow\quad
\dot b\geq cb>0.
\end{equation}
A trajectory cannot cross the level $b=\delta$ downward, and if it starts below $\delta$ it initially moves upward.  Hence $b$ has a positive lower bound depending on the initial point.  Together with boundedness of $\sigma_1$, this gives compact closure in $\GLd$.
\end{proof}

\subsection{Convergence of every bounded orbit}

We use only elementary properties of omega-limit sets and the Lyapunov function.

\begin{lemma}[Bounded orbits converge to equilibria]\label{lem:bounded-converges}
Every bounded forward orbit converges to one of the $2^d$ critical points $W_\epsilon$.
\end{lemma}

\begin{proof}
By \cref{prop:bounded-precompact}, the orbit has compact closure in $\GLd$.  Its omega-limit set
\[
\omega(W_0)
=
\bigcap_{T>0}\overline{\{W(t):t\geq T\}}
\]
is nonempty, compact, invariant, and connected.  Connectedness follows because it is the nested intersection of closures of connected tails of the trajectory.

The decreasing function $F_\beta(W(t))$ has a finite limit on the compact orbit closure.  It is therefore constant on $\omega(W_0)$.  Invariance and the dissipation identity \eqref{eq:Lyapunov} imply
\[
\grad_{g^N}F_\beta=0
\]
at every point of the omega-limit set.  Thus
\[
\omega(W_0)\subset\Crit(F_\beta).
\]
By \cref{thm:2d-critical}, the critical set is finite.  A connected subset of a finite set is a singleton, say $\omega(W_0)=\{W_\epsilon\}$.  A precompact trajectory with a singleton omega-limit set converges to that point.
\end{proof}

\subsection{Measure-zero bounded dynamics}

\begin{theorem}[Generic unboundedness]\label{thm:generic-escape}
The set
\[
\mathcal B
=
\left\{W_0\in\GLd:
\sup_{t\geq0}\normF{W(t;W_0)}<\infty
\right\}
\]
has Lebesgue measure zero.  Hence almost every full-rank initial condition has an unbounded forward orbit.
\end{theorem}

\begin{proof}
By \cref{lem:bounded-converges},
\[
\mathcal B\subset
\bigcup_{\epsilon\in\{-1,1\}^d}W^s(W_\epsilon),
\]
where $W^s(W_\epsilon)$ is the global stable set.  By \cref{cor:hyperbolic} and the stable-manifold theorem \cite{HirschPughShub1977}, the local stable manifold of $W_\epsilon$ is an embedded submanifold of dimension
\[
\frac{d(d+1)}2<d^2.
\]
It therefore has Lebesgue measure zero. The result follows.
\end{proof}

\begin{corollary}[Absolutely continuous random initialization]\label{cor:random-init}
Let $W_0$ have a distribution absolutely continuous with respect to Lebesgue measure on $\Md$.  Conditional on $W_0\in\GLd$, the free-energy orbit is unbounded with probability one.  In particular, if $G$ has independent standard Gaussian entries and $\eps>0$, then
\[
\mathbb P\left(
\sup_{t\geq0}\normF{W(t;\eps G)}=\infty
\right)=1.
\]
\end{corollary}

\begin{proof}
A Gaussian matrix is full rank almost surely, and the exceptional set in \cref{thm:generic-escape} is Lebesgue null.
\end{proof}

\section{The zero-temperature comparison in width two}\label{sec:zero-temperature}

The results above concern the finite-temperature free-energy flow.  We now
compare them with the ordinary deep-linear training flow, obtained formally by
setting $\beta=\infty$.  This comparison is useful for two reasons.  First, it
clarifies the sense in which numerical trajectories can display low-rank
selection even though a full-rank trajectory does not converge to a nonzero
rank-one completion at any fixed positive initialization scale.  Second, it
gives explicit scaled families $W_0^{\eps}=\eps\widehat W_0$ for which every
terminal matrix is full rank but its smaller singular value tends to zero at a
computed rate as $\eps\to0$, while every fixed finite-temperature problem has
generic unbounded dynamics.

Throughout this section $d=2$ and
\[
E_2(W)=\frac12\bigl((W_{11}-1)^2+(W_{22}-1)^2\bigr).
\]
The unregularized reduced balanced flow is
\begin{equation}\label{eq:zero-temp-flow}
\dot W=-\Aop_{N,W}\bigl(E_2'(W)\bigr).
\end{equation}
Its zero-loss set and its rank-one part are
\[
\mathcal Z
=
\left\{
\begin{pmatrix}1&a\\ b&1\end{pmatrix}:a,b\in\mathbb R
\right\},
\qquad
\mathcal H
=
\left\{
\begin{pmatrix}1&\gamma\\ \gamma^{-1}&1\end{pmatrix}:
\gamma\neq0
\right\}.
\]
We write
\[
J=\begin{pmatrix}1&1\\1&1\end{pmatrix},
\qquad
D=\diag(1,-1),
\qquad
DJD=\begin{pmatrix}1&-1\\-1&1\end{pmatrix}
\]
for the two symmetric points of $\mathcal H$.  Let
\[
X_0^N(W)=-\Aop_{N,W}\bigl(E_2'(W)\bigr)
\]
and let $\mathcal T_t^N(W_0)$ denote the solution of
\eqref{eq:zero-temp-flow} with initial condition $W_0$.  Whenever the forward
limit exists, we define the \emph{endpoint map}
\[
\mathcal E_N(W_0)=\lim_{t\to\infty}\mathcal T_t^N(W_0)
\]
 and $\mathcal{E}_N(W_0)$ is the \emph{endpoint} of $W_0$. We also write $W_\infty=\mathcal E_N(W_0)$ when the initial condition is
clear.  Since $\Aop_{N,0}=0$ for $N\geq2$, the origin is an equilibrium and
$\mathcal E_N(0)=0$.  For a fixed matrix $\widehat W_0$, a family
$W_0^{\eps}=\eps\widehat W_0$ will be called a \emph{scaled initialization
family}, and $\eps$ its initialization scale.  Every point of $\mathcal Z$
is an equilibrium of \eqref{eq:zero-temp-flow}.

\subsection{Attraction to the completion manifold and exact rank-one limits}

The completion manifold is normally attracting at each of its full-rank
points.  The rank-one subset has a different status: it is seen in the limit
of small initialization, but it cannot be the exact limit of a trajectory
that starts full rank.

\begin{proposition}[Completion attraction and exclusion of exact rank-one
limits]\label{prop:zero-temp-rank-one}
The following statements hold for every $N\geq2$.
\begin{enumerate}[label=\textup{(\roman*)}]
\item The equilibrium manifold $\mathcal Z\cap\mathrm{GL}_2(\mathbb R)$ is
normally attracting at every point, and its normal contraction is uniform on
compact subsets.
\item If $W_0$ is full rank, then $\mathcal T_t^N(W_0)$ cannot converge to a
point of $\mathcal H$ as $t\to\infty$.
\end{enumerate}
Equivalently, whenever $W_0\in\mathrm{GL}_2(\mathbb R)$ has $\mathcal{E}_N(W_0)\in\mathcal Z$,
\[
\sigma_2\bigl(\mathcal E_N(W_0)\bigr)>0.
\]
\end{proposition}
Thus, if a family of full-rank initial conditions $W_0^{\eps}$ has endpoints in
$\mathcal Z$ converging to a matrix in $\mathcal H$ (as $\varepsilon\rightarrow 0$), then every endpoint remains full rank
but necessarily
\[
\sigma_2\bigl(\mathcal E_N(W_0^{\eps})\bigr)\longrightarrow0.
\]
The results below compute this decay for explicit scaled initialization families.
\begin{proof}
Let $\Pi:\mathbb M_2\to\mathbb M_2$ denote projection onto the diagonal
entries.  At $W_*\in\mathcal Z$, the residual $E_2'(W_*)$ vanishes, and hence
the derivative of the vector field in \eqref{eq:zero-temp-flow} is
\[
DX_0^N(W_*)[H]
=-\Aop_{N,W_*}\bigl(\Pi H\bigr).
\]
If $W_*$ is full rank, $\Aop_{N,W_*}$ is positive definite.  The operator
$-\Aop_{N,W_*}\Pi$ is similar to the negative semidefinite self-adjoint
operator
\[
-\Aop_{N,W_*}^{1/2}\Pi\Aop_{N,W_*}^{1/2}.
\]
Its kernel has dimension two and corresponds to $T_{W_*}\mathcal Z$, while
its two remaining eigenvalues are strictly negative.  The spectral gap is
uniform on compact subsets of $\mathcal Z\cap\mathrm{GL}_2(\mathbb R)$,
which proves part~\textup{(i)}.

For part~\textup{(ii)}, suppose that a full-rank trajectory
converges to
\[
W_\gamma=
\begin{pmatrix}1&\gamma\\ \gamma^{-1}&1\end{pmatrix}
\in\mathcal H.
\]
Write
\[
p(t)=
\begin{pmatrix}W_{11}(t)-1\\ W_{22}(t)-1\end{pmatrix}.
\]
Because the energy gradient is diagonal, the residual satisfies an exact
linear equation with time-dependent coefficients,
\[
\dot p(t)=-M_N(W(t))p(t),
\qquad
(M_N(W))_{ij}
=\left\langle E_{ii},\Aop_{N,W}(E_{jj})\right\rangle_F.
\]
Here $E_{ij}$ denotes the standard matrix unit.  Put
\[
\sigma_\gamma=\left|\gamma+\gamma^{-1}\right|,
\qquad
\alpha_N=2-\frac2N,
\qquad
\eta_\gamma=\frac{N-2}{\sigma_\gamma^2}.
\]
If $W=\sigma uv^T$ has rank one and $P\in\mathbb M_2$, direct evaluation of
the defining sum for $\Aop_{N,W}$ gives
\[
\Aop_{N,W}(P)
=
\sigma^{2-2/N}
\left(
 uu^TP+Pvv^T+(N-2)uu^TPvv^T
\right).
\]
Applying this formula at $W_\gamma$ gives
\[
M_N(W_\gamma)
=
\sigma_\gamma^{\alpha_N}
\begin{pmatrix}
1+\eta_\gamma&\eta_\gamma\\
\eta_\gamma&1+\eta_\gamma
\end{pmatrix}.
\]
Its eigenvalues are
\[
\sigma_\gamma^{\alpha_N},
\qquad
\sigma_\gamma^{\alpha_N}(1+2\eta_\gamma),
\]
and are strictly positive.  Continuity therefore gives constants $c,C>0$
such that
\[
\|p(t)\|\leq Ce^{-ct}
\]
for all sufficiently large $t$.

Let $\sigma_2(t)>0$ be the smaller singular value of $W(t)$.  For large $t$
it is simple; let $u_2(t)$ and $v_2(t)$ be associated unit left and right
singular vectors.  The singular-value equation associated with
\eqref{eq:zero-temp-flow} is
\[
\dot\sigma_2
=-N\sigma_2^{2-2/N}e_2(t),
\qquad
e_2(t)=u_2(t)^TE_2'(W(t))v_2(t).
\]
The preceding exponential estimate implies
\[
|e_2(t)|\leq\|E_2'(W(t))\|_F\leq Ce^{-ct},
\]
so $e_2$ is integrable on a terminal time interval.  If $N=2$, then
\[
\frac{d}{dt}\log\sigma_2(t)=-2e_2(t),
\]
which implies for any $T>0$ that
$$\sigma_2(T) = \sigma_2(0)\exp\left(-2\int_0^T e_2(s)\, ds\right) \geq \sigma_2(0)\exp\left(-2\int_0^\infty e_2(s)\, ds\right)>0,$$
so $\sigma_2(t)$ has a positive limit.  If $N>2$, set
$p_N=(N-2)/N$.  Then
\[
\frac{d}{dt}\sigma_2(t)^{-p_N}=(N-2)e_2(t),
\]
and so 
$$\sigma_2(T)^{-p_N} = \sigma_2(0)^{-p_N} + (N-2)\int_0^T e_2(s)\, ds$$
remains bounded for all $T>0$. Thus $\sigma_2(t)^{-p_N}$ has a finite limit and $\sigma_2(t)$ again cannot
tend to zero.  Both conclusions contradict $W(t)\to W_\gamma$.
\end{proof}

\begin{remark}
The first part of the proposition concerns attraction to the two-dimensional
completion manifold $\mathcal Z$, not attraction to the one-dimensional
hyperbola $\mathcal H$ in the full four-dimensional phase space.  Inside the
rank-one stratum, the diagonal-residual map has rank two and the same positive
normal matrix computed in the proof, so $\mathcal H$ is locally normally
attracting there.  In the full matrix space, however, the additional direction
tangent to $\mathcal Z$ and transverse to $\mathcal H$ is neutral.
\end{remark}

\begin{remark}[Exact rank collapse under block observations]\label{rem:SY-block-collapse}
Part~\textup{(ii)} is a statement about the diagonal observation pattern, not a general rank rigidity of the reduced flow.  For block-diagonal observations with block size $s\geq2$ and $n$ blocks, Shin and Yun prove that the balanced factor flow started from a full-rank deterministic initialization converges to a zero-loss limit whose singular values vanish exactly beyond the $n$th \cite[Theorem 3.3]{ShinYun2026}.  There is no conflict.  In the block case the collapsing directions are the within-block traceless modes, which are observed; on the invariant matrix family used there, the residual component in such a direction equals the vanishing eigenvalue itself, so the singular-value equation degenerates to $\dot\lambda=-N\lambda^{3-2/N}$ and the decay is only algebraic.  The uniform spectral gap of $M_N(W_\gamma)$, which produces the exponential residual decay driving the proof of part~\textup{(ii)}, is therefore absent.  Whether infinite-time rank collapse can occur is thus governed by whether the collapsing directions are seen by the observation pattern.
\end{remark}

\subsection{Exact singular-value and angular coordinates}

We next derive an exact four-dimensional coordinate system in the
positive-determinant component.  Let
\[
R(\vartheta)=
\begin{pmatrix}
\cos\vartheta&-\sin\vartheta\\
\sin\vartheta&\cos\vartheta
\end{pmatrix},
\]
and write, locally away from repeated singular values,
\[
W
=
R(\theta+\delta)
\begin{pmatrix}s&0\\0&t\end{pmatrix}
R(\theta-\delta)^T,
\qquad s>t>0.
\]
Set
\[
u=\frac{s+t}{2},\qquad v=\frac{s-t}{2}.
\]
Then
\begin{equation}\label{eq:zero-temp-coordinate-map}
W=
\begin{pmatrix}
 u\cos2\delta+v\cos2\theta
 &v\sin2\theta-u\sin2\delta\\
 v\sin2\theta+u\sin2\delta
 &u\cos2\delta-v\cos2\theta
\end{pmatrix}.
\end{equation}
Thus $\theta$ records the common orientation of the two singular frames,
while $\delta$ records their mismatch.  In these coordinates,
\begin{equation}\label{eq:zero-temp-energy-coordinates}
E_2(W)
=\bigl(u\cos2\delta-1\bigr)^2+v^2\cos^22\theta.
\end{equation}

Define
\[
A_N(s,t)
=
\sum_{j=0}^{N-1}
 s^{2(N-1-j)/N}t^{2j/N}
=
\frac{s^2-t^2}{s^{2/N}-t^{2/N}},
\]
with the continuous value $A_N(s,s)=Ns^{2-2/N}$.

\begin{proposition}[Exact width-two system]\label{prop:zero-temp-svd-system}
In the preceding chart, the unregularized flow \eqref{eq:zero-temp-flow} is
\begin{equation}\label{eq:zero-temp-svd-system}
\begin{aligned}
\dot s
&=Ns^{2-2/N}
\left[(1-u\cos2\delta)\cos2\delta-v\cos^22\theta\right],\\
\dot t
&=Nt^{2-2/N}
\left[(1-u\cos2\delta)\cos2\delta+v\cos^22\theta\right],\\
\dot\theta
&=\frac{A_N(s,t)}4\sin4\theta,\\
\dot\delta
&=\frac{A_N(s,t)}{2u}
\bigl(u\cos2\delta-1\bigr)\sin2\delta.
\end{aligned}
\end{equation}
In particular, $\theta=\pi/4$ and $\delta=0$ define invariant
hypersurfaces, and their intersection is the symmetric equal-diagonal plane.
\end{proposition}

\begin{proof}
A variation of the singular-value decomposition satisfies
\[
U^T(dW)V
=
\begin{pmatrix}
 ds&2v\,d\theta-2u\,d\delta\\
 2v\,d\theta+2u\,d\delta&dt
\end{pmatrix}.
\]
The eigenvalues of $\Aop_{N,W}$ in the SVD frame are
\[
Ns^{2-2/N},\qquad Nt^{2-2/N},\qquad A_N(s,t),\qquad A_N(s,t).
\]
Consequently, the Riemannian metric is
\[
g^N
=
\frac{ds^2}{Ns^{2-2/N}}
+
\frac{dt^2}{Nt^{2-2/N}}
+
\frac{8v^2}{A_N(s,t)}\,d\theta^2
+
\frac{8u^2}{A_N(s,t)}\,d\delta^2.
\]
Differentiating \eqref{eq:zero-temp-energy-coordinates} and taking its
negative gradient with respect to this diagonal metric gives
\eqref{eq:zero-temp-svd-system}.
\end{proof}

The angular equations in \eqref{eq:zero-temp-svd-system} separate much of the
geometry.  If $0<\theta<\pi/2$, then $\theta$ moves monotonically toward
$\pi/4$.  If $u\cos2\delta<1$ and $|\delta|<\pi/4$, then $|\delta|$ decreases.
Thus $\theta\to\pi/4$ damps the difference of the observed diagonal
entries, while $\delta\to0$ aligns the left and right singular frames.  The
latter angular alignment can coexist, as shown below, with transient
amplification of the antisymmetric matrix coordinate.

\subsection{The symmetric invariant plane}

Let
\[
Q=
\begin{pmatrix}0&1\\1&0\end{pmatrix},
\qquad
\mathcal P=\{xI+zQ:x,z\in\mathbb R\}.
\]
The plane $\mathcal P$ is invariant under \eqref{eq:zero-temp-flow}.  If
\[
\lambda_+=x+z,
\qquad
\lambda_-=x-z
\]
are the two signed eigenvalues, then
\begin{equation}\label{eq:zero-temp-symmetric-eigenvalues}
\dot\lambda_\pm
=N(1-x)|\lambda_\pm|^{2-2/N},
\qquad
x=\frac{\lambda_++\lambda_-}{2}.
\end{equation}
This system can be integrated exactly up to the terminal zero-loss condition
$x_\infty=1$.

\begin{proposition}[Exact depth transition on the symmetric plane]
\label{prop:zero-temp-symmetric-plane}
In each case below, fix $a,b>0$ and assume
$0<\eps<2/(a+b)$.  Let $W_0^{\eps}\in\mathcal P$ be the matrix with the
displayed signed eigenvalues.  Its endpoint exists in each case; write
$W_\infty^{\eps}=\mathcal E_N(W_0^{\eps})$.
\begin{enumerate}[label=\textup{(\roman*)}]
\item Let $N=2$ and
\[
\lambda_+(0)=\eps a,
\qquad
\lambda_-(0)=\eps b,
\qquad a,b>0.
\]
Then
\[
W_\infty^{\eps}
=
\begin{pmatrix}
1&\dfrac{a-b}{a+b}\\[2mm]
\dfrac{a-b}{a+b}&1
\end{pmatrix}.
\]
Thus the endpoint is independent of $\eps$ and, as $a/b$ varies, fills
\[
\left\{
\begin{pmatrix}1&c\\c&1\end{pmatrix}:|c|<1
\right\}.
\]
The closure of this segment is exactly the set of nuclear-norm minimizers in
$\mathcal Z$: for $W\in\mathcal Z$, one has
$\|W\|_*\geq\operatorname{tr}W=2$, with equality precisely for the
positive semidefinite completions.

\item Let $N=2$ and
\[
\lambda_+(0)=\eps a>0,
\qquad
\lambda_-(0)=-\eps b<0.
\]
Then
\[
\lambda_-(\infty)=-\delta_\eps,
\qquad
\lambda_+(\infty)=2+\delta_\eps,
\]
where
\[
\delta_\eps=\sqrt{1+\eps^2ab}-1
=\frac{ab}{2}\eps^2+O(\eps^4).
\]
In particular,
\[
\sigma_2(W_\infty^{\eps})=\delta_\eps
=\frac{ab}{2}\eps^2+O(\eps^4),
\]
and $W_\infty^{\eps}\to J$.

\item Let $N>2$, put $p_N=(N-2)/N$, and suppose
\[
\lambda_+(0)=\eps a,
\qquad
\lambda_-(0)=\eps b,
\qquad a>b>0.
\]
Then
\[
\sigma_2(W_\infty^{\eps})=\lambda_-(\infty)
\sim
\bigl(b^{-p_N}-a^{-p_N}\bigr)^{-1/p_N}\eps,
\qquad
W_\infty^{\eps}\longrightarrow J.
\]
If instead
\[
\lambda_+(0)=\eps a>0,
\qquad
\lambda_-(0)=-\eps b<0,
\]
then
\[
\sigma_2(W_\infty^{\eps})=|\lambda_-(\infty)|
\sim
\bigl(a^{-p_N}+b^{-p_N}\bigr)^{-1/p_N}\eps,
\qquad
W_\infty^{\eps}\longrightarrow J.
\]
Conjugating by $D$, equivalently exchanging the two eigenlines, gives the
corresponding limits toward $DJD$.
\end{enumerate}
For every admissible $\eps>0$ the endpoints in all of these cases are full rank.
\end{proposition}

\begin{proof}
Equation \eqref{eq:zero-temp-symmetric-eigenvalues} shows that the plane is
invariant.  In every case under consideration $x(0)<1$.
On the symmetric plane $\mathcal{P}$, since $2x= \lambda_++\lambda_0 $,
$$\dot{x} = \frac{N}{2} (1-x)\left( |\lambda_-|^{2-2/N} + |\lambda_+|^{2-2/N}\right) $$
which is positive is $x\in(0,1)$ and zero when $x = 1$.

On the invariant hypersurface $\theta = \pi/4$, $\dot{x} = (1-x)Q_N(s,t,\delta)$ where $Q_N(s,t,\delta)$ is a positive function. This also implies that $\dot{x}$ is positive is $x\in(0,1)$ and zero when $x = 1$. It follows that $x(t)$ is nondecreasing and cannot cross $1$. In either case, $x_\infty=1$. Note that since $2x= \lambda_++\lambda_0$, we have that $2= \lambda_+(\infty)+\lambda_0(\infty)$.

When $N=2$ and both eigenvalues are positive,
\[
\frac{d}{dt}\log\lambda_+
=
\frac{d}{dt}\log\lambda_-
=2(1-x),
\]
so their ratio is conserved, and so in the limit we have that $\frac{\lambda_+(\infty)}{\lambda_-(\infty)}=\frac{a}{b}$. This and the identity $\lambda_+(\infty)+\lambda_-(\infty)=2$ derived above gives a system of equations with solutions
$$\lambda_+(\infty) = \frac{2a}{a+b},\;\;\; \lambda_-(\infty) = \frac{2b}{a+b}.$$
Reconstructing $W^\varepsilon_\infty$ using the appropriate eigenbasis we recover the explicit form $W^\varepsilon_\infty$ above and (i) is proved.

If $N=2$ and the signs are opposite, then
\[
\frac{d}{dt}\log\lambda_+=2(1-x),
\qquad
\frac{d}{dt}\log|\lambda_-|=-2(1-x).
\]
Adding both equations we get that $\log \lambda_+|\lambda_-|$ is constant. Using the initial values we obtain that $\lambda_+|\lambda_-|=\eps^2ab$. Writing the terminal values as
$\lambda_-(\infty) = -\delta_\eps$ and $\lambda_+ = 2+\delta_\eps$ gives
\[
\delta_\eps(2+\delta_\eps)=\eps^2ab.
\]
Solving for $\delta_\eps$ finishes the proof of part~\textup{(ii)}.

Suppose now that $N>2$ and define
\[
\psi_N(\lambda)=\sgn(\lambda)|\lambda|^{-p_N}.
\]
A direct differentiation of \eqref{eq:zero-temp-symmetric-eigenvalues} gives
\[
\frac{d}{dt}\psi_N(\lambda_\pm)=-(N-2)(1-x).
\]
Thus
\[
\psi_N(\lambda_+)-\psi_N(\lambda_-)
\]
is conserved.  In the positive-positive case, if
$t_\eps=\lambda_-(\infty)$, then
\[
t_\eps^{-p_N}-(2-t_\eps)^{-p_N}
=
\eps^{-p_N}\bigl(b^{-p_N}-a^{-p_N}\bigr).
\]
This implies the first asymptotic in part~\textup{(iii)}.  In the
positive-negative case, if
$\delta_\eps=|\lambda_-(\infty)|$, then
\[
(2+\delta_\eps)^{-p_N}+\delta_\eps^{-p_N}
=
\eps^{-p_N}\bigl(a^{-p_N}+b^{-p_N}\bigr),
\]
which gives the second asymptotic.  The matrix limits follow from the fixed
eigenvectors of $\mathcal P$.
\end{proof}

In terms of the endpoint map, the preceding formulas quantify its behavior at
the origin.  Put $p_N=(N-2)/N$.  In the same-sign family of
part~\textup{(iii)},
\[
\sigma_2\bigl(\mathcal E_N(W_0^{\eps})\bigr)
\sim
\bigl(b^{-p_N}-a^{-p_N}\bigr)^{-1/p_N}\eps,
\]
while in its mixed-sign family,
\[
\sigma_2\bigl(\mathcal E_N(W_0^{\eps})\bigr)
\sim
\bigl(a^{-p_N}+b^{-p_N}\bigr)^{-1/p_N}\eps.
\]
In both cases $W_0^{\eps}\to0$ and
$\mathcal E_N(W_0^{\eps})\to J$; exchanging the two eigenlines gives the
corresponding limit $DJD$.  In the mixed-sign depth-two family of
part~\textup{(ii)},
\[
\sigma_2\bigl(\mathcal E_2(W_0^{\eps})\bigr)
=\frac{ab}{2}\eps^2+O(\eps^4),
\qquad
\mathcal E_2(W_0^{\eps})\longrightarrow J.
\]
Since $\mathcal E_N(0)=0$, these scaled families prove that $\mathcal E_N$ is
discontinuous at the origin.  By contrast, part~\textup{(i)}
shows that the depth-two same-sign families have scale-independent full-rank
endpoints rather than rank collapse.

\begin{remark}[Relation to the results of Shin and Yun]\label{rem:SY-correspondence}
The same-sign cases of \cref{prop:zero-temp-symmetric-plane} recover, in width two, the diagonal-observation case of \cite[Theorem 3.3]{ShinYun2026}, which precedes this paper.  Shin and Yun analyze the factor gradient flow with every layer initialized to the symmetric matrix $(\alpha-\alpha/m)I_d+(\alpha/m)J_d$ for $\alpha>0$ and $m>1$.  This initialization is balanced, so the end-to-end trajectory solves the reduced flow \eqref{eq:zero-temp-flow}; at $d=2$ with unit diagonal targets it lies on the plane $\mathcal P$ with
\[
\lambda_+(0)=\Bigl(\frac{\alpha(m+1)}{m}\Bigr)^{\!N},
\qquad
\lambda_-(0)=\Bigl(\frac{\alpha(m-1)}{m}\Bigr)^{\!N}.
\]
Their implicit equations for the limiting singular values \cite[equations (8) and (9)]{ShinYun2026} are precisely the conservation of $\psi_N(\lambda_+)-\psi_N(\lambda_-)$ used in the proof above, their depth-two closed form is the endpoint of part~\textup{(i)}, and their decoupled case $m=\infty$ is the degenerate ray $a=b$, whose endpoint is the identity.  Writing $\operatorname{srank}(W)=\|W\|_F^2/\|W\|_{\mathrm{op}}^2$, their Corollary~3.4 states that the stable rank of the limit tends to one as $\alpha\to0$.  In width two, their end-to-end initialization scale is $\eps=\alpha^N$ up to constants depending on $m$ and $N$; part~\textup{(iii)} therefore gives a terminal smaller singular value of order $\alpha^N$, while every fixed-$\alpha$ endpoint remains full rank, consistently with \cref{prop:zero-temp-rank-one}.  Since the ratio $\lambda_+(0)/\lambda_-(0)=\bigl((m+1)/(m-1)\bigr)^N$ covers all of $(1,\infty)$, their family realizes every same-sign initialization on $\mathcal P$ up to scale.  The opposite-sign cases, in which the determinant is negative, and \cref{thm:zero-temp-equal-diagonal} have no analogue in \cite{ShinYun2026}.  The derivation here also gives their algebraic computation a geometric interpretation: the implicit equations are conservation laws of the Riemannian system \eqref{eq:zero-temp-symmetric-eigenvalues}.  At depth two, the positive-determinant segment law of part~\textup{(i)} should also be compared with the connectivity framework of Bai, Zhao, and Zhang, in which depth-two completion with disconnected observations whose components are complete bipartite selects minimal-nuclear-norm solutions under additional assumptions \cite{BaiZhaoZhang2024}.
\end{remark}

The proposition gives a rigorous form of the depth-two-versus-deeper
transition observed numerically.  At depth two, positive determinant can
retain two singular values of order one as $\eps\to0$.  At every depth
$N>2$, the smaller terminal singular value is of order $\eps$ on this plane.
For negative determinant, depth two instead produces the smaller scale
$\eps^2$.

\subsection{A nonlinear nonsymmetric perturbation}

The symmetric plane sets both angular variables to their invariant values.
The next result allows a fixed mismatch between the left and right singular
frames.  It is therefore a genuinely nonsymmetric perturbation, although the
two observed diagonal entries remain equal.

\begin{theorem}[Equal-diagonal perturbations]\label{thm:zero-temp-equal-diagonal}
Fix $N>2$, $a>b>0$, and
$\delta_0\in(-\pi/4,\pi/4)$.  Let $W_0^{\eps}$ be the matrix whose
coordinates in \eqref{eq:zero-temp-coordinate-map} satisfy
\[
s(0)=\eps a,
\qquad
t(0)=\eps b,
\qquad
\theta(0)=\frac\pi4,
\qquad
\delta(0)=\delta_0.
\]
Assume
\[
0<\eps<\frac{2}{(a+b)\cos2\delta_0}.
\]
Then the solution converges to a full-rank zero-loss completion
$W_{\eps,\infty}=\mathcal E_N(W_0^{\eps})$.  If $p_N=(N-2)/N$, then
\[
t_\infty
=\sigma_2(W_{\eps,\infty})
\sim
\bigl(b^{-p_N}-a^{-p_N}\bigr)^{-1/p_N}\eps.
\]
Moreover, there is a constant $C_{N,a,b}>0$ such that
\[
\sin2\delta_\infty
=
C_{N,a,b}\sin2\delta_0\,
\eps^{2/N}\bigl(1+o(1)\bigr).
\]
Consequently,
\[
W_{\eps,\infty}\longrightarrow J
\qquad\text{as }\eps\to0,
\]
and
\[
(W_{\eps,\infty})_{21}-(W_{\eps,\infty})_{12}
=
2C_{N,a,b}\sin2\delta_0\,
\eps^{2/N}\bigl(1+o(1)\bigr).
\]
In particular, when $\sin2\delta_0\neq0$, the terminal asymmetry has exact
order $\eps^{2/N}$.
\end{theorem}

\begin{proof}
The hypersurface $\theta=\pi/4$ is invariant by
\eqref{eq:zero-temp-svd-system}.  On this hypersurface, put
$x=u\cos2\delta$.  The displayed bound on $\eps$ gives $x(0)<1$.  Since
$E_2=(x-1)^2$ there, energy dissipation implies that $x(t)$ increases and
cannot cross $1$.  Also $|\delta(t)|$ decreases, so
$\cos2\delta(t)\geq\cos2\delta_0>0$.  It follows that $s$ and $t$ increase,
while
\[
u(t)=\frac{x(t)}{\cos2\delta(t)}\leq\frac1{\cos2\delta_0}.
\]
Thus the trajectory is bounded and remains full rank for every finite time.
Its monotone coordinates have limits.  If $x_\infty<1$, the equation for
$\dot s$ would remain bounded below by a positive constant, contradicting
boundedness.  Hence $x_\infty=1$, so the terminal point has zero loss.

On $\theta=\pi/4$, the $s$ and $t$ equations imply the exact conservation
law
\begin{equation}\label{eq:zero-temp-equal-diagonal-invariant}
t^{-p_N}-s^{-p_N}
=
\eps^{-p_N}\bigl(b^{-p_N}-a^{-p_N}\bigr).
\end{equation}
Since $x_\infty=1$ implies $u_\infty\geq1$, one has
$s_\infty\geq1$.  It follows from
\eqref{eq:zero-temp-equal-diagonal-invariant} that
\[
t_\infty
\sim
\bigl(b^{-p_N}-a^{-p_N}\bigr)^{-1/p_N}\eps.
\]

For $\delta_0\neq0$, the angular equation and the $s$ equation give
\begin{equation}\label{eq:zero-temp-delta-versus-s}
\frac{d}{ds}\log|\sin2\delta|
=-\frac{2A_N(s,t)}{N(s+t)s^{2-2/N}}.
\end{equation}
The case $\delta_0=0$ follows by invariance.  Put
\[
c_0=b^{-p_N}-a^{-p_N},
\qquad
h(q)=\bigl(q^{-p_N}+c_0\bigr)^{-1/p_N}.
\]
Equation \eqref{eq:zero-temp-equal-diagonal-invariant} gives
$t=\eps h(s/\eps)$.  By the homogeneity of $A_N$, integration of
\eqref{eq:zero-temp-delta-versus-s} yields
\[
\log\frac{|\sin2\delta_0|}{|\sin2\delta_\infty|}
=
\int_a^{s_\infty/\eps}
\frac{2A_N(q,h(q))}{N(q+h(q))q^{2-2/N}}\,dq.
\]
The integrand satisfies
\[
\frac{2A_N(q,h(q))}{N(q+h(q))q^{2-2/N}}
=
\frac{2}{Nq}+O(q^{-1-2/N})
\qquad(q\to\infty).
\]
Consequently, the integral equals
\[
\frac2N\log\frac1\eps+O(1).
\]
It follows that $\delta_\infty\to0$.  Since
$x_\infty=u_\infty\cos2\delta_\infty=1$, we obtain
$u_\infty\to1$, and the preceding asymptotic for $t_\infty$ gives
$s_\infty\to2$.  Subtracting $2/(Nq)$ from the integrand shows that the
$O(1)$ term has a finite limit.  More precisely,
\[
L_{N,a,b}
=
\int_a^\infty
\left[
\frac{2A_N(q,h(q))}{N(q+h(q))q^{2-2/N}}
-
\frac2{Nq}
\right]dq
\]
converges, and
\[
C_{N,a,b}
=
\exp(-L_{N,a,b})\left(\frac a2\right)^{2/N}>0.
\]
This proves the asserted asymptotic for $\sin2\delta_\infty$.

Finally, \eqref{eq:zero-temp-coordinate-map} with $\theta=\pi/4$ gives
\[
(W_{\eps,\infty})_{21}-(W_{\eps,\infty})_{12}
=2u_\infty\sin2\delta_\infty.
\]
Since $u_\infty\to1$, the asymmetry estimate and the convergence to $J$
follow.
\end{proof}

At depth three, when $\sin2\delta_0\neq0$, the theorem gives the two
distinct endpoint scales
\[
\sigma_2(W_{\eps,\infty})\asymp\eps,
\qquad
|(W_{\eps,\infty})_{21}-(W_{\eps,\infty})_{12}|
\asymp\eps^{2/3}.
\]
Thus the endpoint can be much less symmetric than its smallest singular value
would suggest, while still converging to the symmetric rank-one matrix $J$.

\section{Interpretation and open problems}
\label{sec:discussion}

\subsection{Failure of an equilibrium selection principle}

The original thermodynamic proposal was that the minimizing set
\[
\operatorname*{argmin}_{W\in\GLd}F_\beta(W)
\]
might select among a degenerate family of energy minimizers, especially as $\beta\to\infty$ \cite{Menon2025,MenonYu2026}.  For the diagonal-completion loss, the proposed set is empty for every finite $\beta$.  The stronger dynamical conclusions are:
\[
\begin{array}{c}
\text{the only finite full-rank equilibria are saddles},\\[0.3em]
\text{finite rank-deficient limits are impossible},\\[0.3em]
\text{and generic trajectories are unbounded}.
\end{array}
\]
Thus fiber entropy does not select one completion from the affine zero-loss set.  It creates drift along noncompact directions that the loss does not control.

The obstruction is not peculiar to the diagonal observation pattern.  By \cref{thm:linear-sensing-dichotomy}, every noninjective squared linear matrix-sensing loss on the unrestricted matrix space has unbounded free energy.  The criterion in \cref{prop:Theta-criterion} gives the more intrinsic statement: entropic minimization fails whenever a finite energy sublevel carries arbitrarily large fiber entropy.  Diagonal completion realizes the strongest version of this pathology, since the entropy is already unbounded on the exact minimizing set.

This does not make the entropy irrelevant.  It identifies a different mechanism: after the energy has nearly minimized the observed coordinates, the entropy acts on the flat unobserved degrees of freedom.  In this example, that action is destabilizing rather than selecting.

\subsection{Noncommuting limits and metastability}

The zero-temperature analysis in \cref{sec:zero-temperature} makes one side of the comparison rigorous.  The endpoint map $\mathcal E_N$ is defined by the infinite-time limit of the unregularized flow whenever that limit exists.  By \cref{prop:zero-temp-rank-one}, a full-rank initial condition with an endpoint in $\mathcal Z$ always satisfies $\sigma_2(\mathcal E_N(W_0))>0$.  Nevertheless, on the symmetric invariant plane and on the nonsymmetric equal-diagonal family of \cref{thm:zero-temp-equal-diagonal}, scaled full-rank initializations $W_0^{\eps}=\eps\widehat W_0$ satisfy $\mathcal E_N(W_0^{\eps})\to J$ or $DJD$ as $\eps\to0$.  For $N>2$ the smaller terminal singular value is asymptotic to the explicit positive multiples of $\eps$ in \cref{prop:zero-temp-symmetric-plane,thm:zero-temp-equal-diagonal}, while at $N=2$ the positive-determinant symmetric family retains a nondegenerate rank-two limit and, in the notation of \cref{prop:zero-temp-symmetric-plane}, the mixed-sign family satisfies $\sigma_2(\mathcal E_2(W_0^{\eps}))=(ab/2)\eps^2+O(\eps^4)$.

This is compatible with the finite-temperature theorem because the two procedures ask different long-time questions.  Taking the zero-temperature system first and then sending $t\to\infty$ produces a finite completion whose subsequent small-initialization limit may be rank one.  At any fixed finite $\beta$, by contrast, almost every orbit is unbounded.  Thus low-rank selection by ordinary training and eventual entropy-driven escape are not competing descriptions of the same fixed dynamical system.

A genuine metastability theorem remains open.  For large finite $\beta$, one expects a trajectory first to shadow the unregularized flow, enter a neighborhood of the completion selected at $\beta=\infty$, and only later depart along an entropy-driven direction.  Such a theorem should quantify
\[
T_{\mathrm{escape}}(\beta,W_0,N,d)
\]
and relate it to the initialization scale.  The rank-boundary proof suggests that the closest approach to a corank-one state should depend sensitively on $\beta$.  At corank at least two, the logarithmic barrier may permit exponentially small singular values before the entropy term becomes dynamically decisive.

\subsection{The remaining zero-temperature width-two problem}

The exact calculations in \cref{sec:zero-temperature} do not yet determine the pushforward of Gaussian end-to-end initialization under the endpoint map $\mathcal E_N$.  The positive-determinant SVD system identifies two stable angular mechanisms: unequal diagonal entries are driven toward $\theta=\pi/4$, and the mismatch of the left and right singular frames is damped while the observed loss is decreasing.  The feedback of the first angular displacement into the singular-value competition begins quadratically.  These facts explain why the rank-one limit persists on the invariant families treated above, but they do not yet give a global basin theorem.

A complete solution requires a global analysis of the angular phase portrait, a corresponding singular-vector chart with one signed singular value for the negative-determinant component, and a description of the separatrices between the branches tending toward $J$ and $DJD$.  The final probabilistic step is then to determine the law of $\mathcal E_N(\eps G)$ for a Gaussian matrix $G$, including its conditional behavior on the two determinant components.  At depth two and positive determinant, this should identify the nondegenerate law on the minimal-nuclear-norm segment.  At depths $N>2$, it should determine the probabilities and rates of concentration near the two rank-one completions.

\subsection{What generic unboundedness does not yet say}

The conclusion
\[
\sup_{t\geq0}\normF{W(t)}=\infty
\]
does not determine an escape rate or asymptotic direction.  It also does not prove that the norm tends to infinity without large returns.  Natural next questions are:

\begin{enumerate}
\item Does $\normF{W(t)}\to\infty$ for almost every initial condition?
\item Does the completion loss satisfy $E_d(W(t))\to0$ along generic escaping trajectories?
\item Does $W(t)/\normF{W(t)}$ converge, and if so, what ranks and sign patterns can occur?
\item Can one classify singular escape at infinity, where some singular values diverge and others vanish?
\end{enumerate}

\subsection{The finite-temperature depth-two problem}

The depth-two results in \cref{sec:zero-temperature} concern the unregularized system.  At finite temperature, unboundedness of the free energy remains true in every width, but the equilibrium and rank-boundary arguments used in the main theorem change.  For $d\geq3$, the entropy is still strictly concave in singular-value coordinates, while its derivative at a one-zero boundary point is finite and the energy and entropy terms have the same order in the smallest singular value.  Consequently, the proof of existence in every sign chamber and the universal boundary repulsion both require modification.  Width two is additionally exceptional because the singular-value Hessian has a kernel and critical manifolds appear.

A complete finite-temperature depth-two classification in arbitrary width is therefore a separate problem.  It may involve temperature-dependent loss of interior equilibria through the corank-one boundary.

\subsection{Adding confinement and recovering a selection problem}

The general criteria of \cref{sec:general-energy} identify a hierarchy of hypotheses.  To make $F_\beta$ merely bounded below, the energy must dominate the logarithmic entropy growth.  To prevent escape in norm, strict domination is needed.  To guarantee an actual minimizer inside the open manifold $\GLd$, one must additionally prevent minimizing sequences from approaching the corank-one boundary, where the full-rank entropy has a finite limit.

A coercive penalty such as
\[
E_{d,\lambda}(W)
=
E_d(W)+\frac\lambda2\normF{W}^2,
\qquad \lambda>0,
\]
is more than sufficient at infinity: the quadratic term dominates the logarithmic entropy for every finite $\beta$.  More generally, \cref{cor:norm-growth-criteria} shows that any superlogarithmically coercive energy confines the free energy in norm.  The ridge term does not itself create a barrier at rank loss, however, so it does not by itself guarantee that the infimum is attained in $\GLd$.  One may instead work on a closed full-rank region such as $\{\sigma_d\geq\eta\}$, add a barrier at corank one, or formulate compatible entropies on all rank strata.

Once compactness and attainment are available, the usual low-temperature mechanism becomes rigorous.  Proposition~\ref{prop:low-temp-selection} shows that every limit point of minimizers of $E-\beta^{-1}S_N$ as $\beta\to\infty$ maximizes $S_N$ on the minimizing set of $E$.  Thus entropy can genuinely select among degenerate minimizers when that set is compact, the entropy is continuous there, and the free-energy minimizers remain precompact.  The failure for diagonal completion is not a failure of this variational argument.  It is the failure of its compactness hypotheses: the minimizing set is noncompact and $S_N$ has no finite maximizer on it.


\section{Proof of the main theorem}

\begin{proof}[Proof of \cref{thm:main}]
Part (i) is \cref{thm:unbounded}.  Part (ii) is \cref{thm:critical-diagonal,thm:2d-critical}.  Part (iii) is \cref{thm:Morse-index,cor:hyperbolic}.  Part (iv) is \cref{cor:no-singular-limit}.  Part (v) is \cref{thm:forward-complete}.  Part (vi) follows from \cref{lem:bounded-converges,thm:generic-escape}.
\end{proof}

\appendix

\section{A direct derivation of the off-diagonal Hessian block}\label{app:hessian}

For completeness, we give the $2\times2$ calculation behind \eqref{eq:hessian-offdiag}.  Let
\[
X=\begin{pmatrix}x&0\\0&y\end{pmatrix},
\qquad
H=\begin{pmatrix}0&a\\b&0\end{pmatrix},
\qquad x\neq y>0.
\]
Then
\[
(X+tH)^T(X+tH)
=
\begin{pmatrix}
 x^2+t^2b^2&t(xa+yb)\\
 t(xa+yb)&y^2+t^2a^2
\end{pmatrix}.
\]
Second-order eigenvalue perturbation gives
\begin{align*}
\sigma_x''(0)
&=
\frac1x\left(
 b^2+\frac{(xa+yb)^2}{x^2-y^2}
\right),\\
\sigma_y''(0)
&=
\frac1y\left(
 a^2-\frac{(xa+yb)^2}{x^2-y^2}
\right).
\end{align*}
If $g_x=\partial_xf$ and $g_y=\partial_yf$, the first singular-value derivatives vanish, so
\[
D^2(f\circ\sigma)(X)[H,H]
=g_x\sigma_x''(0)+g_y\sigma_y''(0).
\]
Algebra gives
\begin{align*}
D^2(f\circ\sigma)(X)[H,H]
&=
\frac12\frac{g_x-g_y}{x-y}(a+b)^2\\
&\quad+
\frac12\frac{g_x+g_y}{x+y}(a-b)^2.
\end{align*}
Passing from $X$ to a signed diagonal matrix replaces $(a,b)$ by $(\epsilon_i a,\epsilon_jb)$, which is \eqref{eq:hessian-offdiag}.

\section{Uniform corank-one repulsion on bounded sets}\label{app:uniform-repulsion}

We record the quantitative estimate used in \cref{prop:bounded-precompact}.  Suppose
\[
m\leq\sigma_j\leq M,
\qquad j<d,
\qquad
0<b=\sigma_d<m/2.
\]
From \eqref{eq:boundary-drift-exact},
\begin{align*}
N b^{2-2/N}g_d
&=
\sum_{j<d}
\left[
\frac{b}{\sigma_j^{2/N}-b^{2/N}}
-
\frac{N b^{3-2/N}}{\sigma_j^2-b^2}
\right]\\
&\geq
\frac{d-1}{M^{2/N}}b
-C_1b^{1+2/N}-C_2b^{3-2/N}.
\end{align*}
Also,
\[
-N e_d b^{2-2/N}\geq-C_3b^{2-2/N}
\]
on a bounded set.  Since every exponent in the three error terms is strictly greater than $1$ when $N>2$, there is $\delta>0$ such that
\[
\dot b\geq\frac{d-1}{2\beta M^{2/N}}b
\qquad
\text{whenever }0<b\leq\delta.
\]
This is the outward barrier \eqref{eq:outward-barrier}.

\bibliographystyle{amsalpha}
\bibliography{biblio}

@inproceedings{BaiZhaoZhang2024,
  author    = {Bai, Zhiwei and Zhao, Jiajie and Zhang, Yaoyu},
  title     = {Connectivity shapes implicit regularization in matrix factorization models for matrix completion},
  booktitle = {Advances in Neural Information Processing Systems 37},
  year      = {2024},
}

@inproceedings{ShinYun2026,
  author    = {Shin, Baekrok and Yun, Chulhee},
  title     = {Implicit bias and loss of plasticity in matrix completion: Depth promotes low-rankness},
  booktitle = {International Conference on Learning Representations},
  year      = {2026},
  note      = {arXiv:2603.04703},
}

@inproceedings{AroraCohenHazan2018,
  author    = {Arora, Sanjeev and Cohen, Nadav and Hazan, Elad},
  title     = {On the Optimization of Deep Networks: Implicit Acceleration by Overparameterization},
  booktitle = {Proceedings of the 35th International Conference on Machine Learning},
  series    = {Proceedings of Machine Learning Research},
  volume    = {80},
  pages     = {244--253},
  year      = {2018},
  url       = {https://proceedings.mlr.press/v80/arora18a.html}
}

@inproceedings{AroraCohenHuLuo2019,
  author    = {Arora, Sanjeev and Cohen, Nadav and Hu, Wei and Luo, Yuping},
  title     = {Implicit Regularization in Deep Matrix Factorization},
  booktitle = {Advances in Neural Information Processing Systems 32},
  pages     = {7411--7422},
  year      = {2019},
  url       = {https://proceedings.neurips.cc/paper/2019/hash/c0c783b5fc0d7d808f1d14a6e9c8280d-Abstract.html}
}

@article{BahEtAl2022,
  author  = {Bah, Bubacarr and Rauhut, Holger and Terstiege, Ulrich and Westdickenberg, Michael},
  title   = {Learning Deep Linear Neural Networks: {R}iemannian Gradient Flows and Convergence to Global Minimizers},
  journal = {Information and Inference: A Journal of the IMA},
  volume  = {11},
  number  = {1},
  pages   = {307--353},
  year    = {2022},
  doi     = {10.1093/imaiai/iaaa039}
}

@misc{ChenKotwalMenon2025,
  author        = {Chen, Alan and Kotwal, Tejas S. and Menon, Govind},
  title         = {Entropic Regularization in the Deep Linear Network},
  year          = {2025},
  eprint        = {2512.06137},
  note        = {arxiv: 2512.06137},
  archiveprefix = {arXiv},
  primaryclass  = {cs.NE},
  url           = {https://arxiv.org/abs/2512.06137}
}

@article{CohenMenonVeraszto2023,
  author  = {Cohen, Nadav and Menon, Govind and Veraszto, Zsolt},
  title   = {Deep Linear Networks for Matrix Completion---an Infinite Depth Limit},
  journal = {SIAM Journal on Applied Dynamical Systems},
  volume  = {22},
  number  = {4},
  pages   = {3208--3232},
  year    = {2023},
  doi     = {10.1137/22M1530653}
}

@misc{DrusvyatskiyKempton2015,
  author        = {Drusvyatskiy, Dmitriy and Kempton, Courtney},
  title         = {Variational Analysis of Spectral Functions Simplified},
  year          = {2015},
  eprint        = {1506.05170},
  note        = {arxiv:1506.05170},
  archiveprefix = {arXiv},
  primaryclass  = {math.OC},
  url           = {https://arxiv.org/abs/1506.05170}
}

@inproceedings{GunasekarEtAl2017,
  author    = {Gunasekar, Suriya and Woodworth, Blake and Bhojanapalli, Srinadh and Neyshabur, Behnam and Srebro, Nathan},
  title     = {Implicit Regularization in Matrix Factorization},
  booktitle = {Advances in Neural Information Processing Systems 30},
  year      = {2017},
  url       = {https://proceedings.neurips.cc/paper/2017/hash/58191d2a914c6dae66371c9dcdc91b41-Abstract.html}
}

@book{HirschPughShub1977,
  author    = {Hirsch, Morris W. and Pugh, Charles C. and Shub, Michael},
  title     = {Invariant Manifolds},
  series    = {Lecture Notes in Mathematics},
  volume    = {583},
  publisher = {Springer-Verlag},
  address   = {Berlin},
  year      = {1977},
  doi       = {10.1007/BFb0092042}
}

@article{Lewis1996,
  author  = {Lewis, Adrian S.},
  title   = {Derivatives of Spectral Functions},
  journal = {Mathematics of Operations Research},
  volume  = {21},
  number  = {3},
  pages   = {576--588},
  year    = {1996},
  doi     = {10.1287/moor.21.3.576}
}

@article{LewisSendov2001,
  author  = {Lewis, Adrian S. and Sendov, Hristo S.},
  title   = {Twice Differentiable Spectral Functions},
  journal = {SIAM Journal on Matrix Analysis and Applications},
  volume  = {23},
  number  = {2},
  pages   = {368--386},
  year    = {2001},
  doi     = {10.1137/S089547980036838X}
}

@incollection{Menon2025,
  author    = {Menon, Govind},
  title     = {The Geometry of the Deep Linear Network},
  booktitle = {XIV Symposium on Probability and Stochastic Processes},
  series    = {Progress in Probability},
  publisher = {Birkh{\"a}user},
  year      = {2025},
  pages     = {1--47},
  note      = {arXiv:2411.09004},
  url       = {https://arxiv.org/abs/2411.09004}
}

@misc{MenonYu2026,
  author        = {Menon, Govind and Yu, Tianmin},
  title         = {An Entropy Formula for the Deep Linear Network},
  year          = {2025},
  eprint        = {2509.09088},
  archiveprefix = {arXiv},
  primaryclass  = {cs.LG},
  note          = {to appear in SIAM Journal on Mathematical Analysis},
  url           = {https://arxiv.org/abs/2509.09088}
}

@article{AlbertAnderson1984,
  author  = {Albert, A. and Anderson, J. A.},
  title   = {On the Existence of Maximum Likelihood Estimates in Logistic Regression Models},
  journal = {Biometrika},
  volume  = {71},
  number  = {1},
  pages   = {1--10},
  year    = {1984},
  doi     = {10.1093/biomet/71.1.1}
}

@article{CandesRecht2009,
  author  = {Cand\`es, Emmanuel J. and Recht, Benjamin},
  title   = {Exact Matrix Completion via Convex Optimization},
  journal = {Foundations of Computational Mathematics},
  volume  = {9},
  number  = {6},
  pages   = {717--772},
  year    = {2009},
  doi     = {10.1007/s10208-009-9045-5}
}

@article{FriedmanHastieTibshirani2008,
  author  = {Friedman, Jerome and Hastie, Trevor and Tibshirani, Robert},
  title   = {Sparse Inverse Covariance Estimation with the Graphical Lasso},
  journal = {Biostatistics},
  volume  = {9},
  number  = {3},
  pages   = {432--441},
  year    = {2008},
  doi     = {10.1093/biostatistics/kxm045}
}

@article{FriedmanHastieTibshirani2010,
  author  = {Friedman, Jerome and Hastie, Trevor and Tibshirani, Robert},
  title   = {Regularization Paths for Generalized Linear Models via Coordinate Descent},
  journal = {Journal of Statistical Software},
  volume  = {33},
  number  = {1},
  pages   = {1--22},
  year    = {2010},
  doi     = {10.18637/jss.v033.i01}
}

@article{Huber1964,
  author  = {Huber, Peter J.},
  title   = {Robust Estimation of a Location Parameter},
  journal = {The Annals of Mathematical Statistics},
  volume  = {35},
  number  = {1},
  pages   = {73--101},
  year    = {1964},
  doi     = {10.1214/aoms/1177703732}
}

@article{RechtFazelParrilo2010,
  author  = {Recht, Benjamin and Fazel, Maryam and Parrilo, Pablo A.},
  title   = {Guaranteed Minimum-Rank Solutions of Linear Matrix Equations via Nuclear Norm Minimization},
  journal = {SIAM Review},
  volume  = {52},
  number  = {3},
  pages   = {471--501},
  year    = {2010},
  doi     = {10.1137/070697835}
}

@article{SoudryEtAl2018,
  author  = {Soudry, Daniel and Hoffer, Elad and Nacson, Mor Shpigel and Gunasekar, Suriya and Srebro, Nathan},
  title   = {The Implicit Bias of Gradient Descent on Separable Data},
  journal = {Journal of Machine Learning Research},
  volume  = {19},
  number  = {70},
  pages   = {1--57},
  year    = {2018},
  url     = {https://jmlr.org/papers/v19/18-188.html}
}

@article{GavishDonoho2017,
  author  = {Gavish, Matan and Donoho, David L.},
  title   = {Optimal Shrinkage of Singular Values},
  journal = {IEEE Transactions on Information Theory},
  volume  = {63},
  number  = {4},
  pages   = {2137--2152},
  year    = {2017},
  doi     = {10.1109/TIT.2017.2653801}
}

@article{CandesPlan2011,
  author  = {Cand\`es, Emmanuel J. and Plan, Yaniv},
  title   = {Tight Oracle Inequalities for Low-Rank Matrix Recovery from a Minimal Number of Noisy Random Measurements},
  journal = {IEEE Transactions on Information Theory},
  volume  = {57},
  number  = {4},
  pages   = {2342--2359},
  year    = {2011},
  doi     = {10.1109/TIT.2011.2111771}
}

@article{HoerlKennard1970,
  author  = {Hoerl, Arthur E. and Kennard, Robert W.},
  title   = {Ridge Regression: Biased Estimation for Nonorthogonal Problems},
  journal = {Technometrics},
  volume  = {12},
  number  = {1},
  pages   = {55--67},
  year    = {1970},
  doi     = {10.1080/00401706.1970.10488634}
}

@article{KoltchinskiiLouniciTsybakov2011,
  author  = {Koltchinskii, Vladimir and Lounici, Karim and Tsybakov, Alexandre B.},
  title   = {Nuclear-Norm Penalization and Optimal Rates for Noisy Low-Rank Matrix Completion},
  journal = {The Annals of Statistics},
  volume  = {39},
  number  = {5},
  pages   = {2302--2329},
  year    = {2011},
  doi     = {10.1214/11-AOS894}
}

@article{CortesVapnik1995,
  author  = {Cortes, Corinna and Vapnik, Vladimir},
  title   = {Support-Vector Networks},
  journal = {Machine Learning},
  volume  = {20},
  number  = {3},
  pages   = {273--297},
  year    = {1995},
  doi     = {10.1007/BF00994018}
}

@article{LangeLittleTaylor1989,
  author  = {Lange, Kenneth L. and Little, Roderick J. A. and Taylor, Jeremy M. G.},
  title   = {Robust Statistical Modeling Using the {$t$} Distribution},
  journal = {Journal of the American Statistical Association},
  volume  = {84},
  number  = {408},
  pages   = {881--896},
  year    = {1989},
  doi     = {10.1080/01621459.1989.10478852}
}

@article{BeatonTukey1974,
  author  = {Beaton, Albert E. and Tukey, John W.},
  title   = {The Fitting of Power Series, Meaning Polynomials, Illustrated on Band-Spectroscopic Data},
  journal = {Technometrics},
  volume  = {16},
  number  = {2},
  pages   = {147--185},
  year    = {1974},
  doi     = {10.1080/00401706.1974.10489171}
}

@inproceedings{GiampourasEtAl2020,
  author    = {Giampouras, Paris and Vidal, Ren\'e and Rontogiannis, Athanasios A. and Haeffele, Benjamin D.},
  title     = {A Novel Variational Form of the {S}chatten-{$p$} Quasi-Norm},
  booktitle = {Advances in Neural Information Processing Systems 33},
  pages     = {21453--21463},
  year      = {2020},
  url       = {https://proceedings.neurips.cc/paper/2020/hash/f53eb4122d5e2ce81a12093f8f9ce922-Abstract.html}
}

@inproceedings{KobayashiAkramVonOswald2024,
  author    = {Kobayashi, Seijin and Akram, Yassir and von Oswald, Johannes},
  title     = {Weight Decay Induces Low-Rank Attention Layers},
  booktitle = {Advances in Neural Information Processing Systems 37},
  pages     = {4481--4510},
  year      = {2024},
  doi       = {10.52202/079017-0146}
}

\end{document}